\documentclass[a4paper,12pt]{article}%
\usepackage{amssymb}
\usepackage{amsfonts}
\usepackage{amsmath}
\usepackage{authblk}
\usepackage[nohead]{geometry}
\usepackage[singlespacing]{setspace}
\usepackage[bottom]{footmisc}
\usepackage{indentfirst}
\usepackage{endnotes}
\usepackage{graphicx}%
\usepackage{rotating}
\usepackage{amsthm}
\usepackage{booktabs}
\usepackage{tikz}
\usepackage{algorithm}
\usepackage{comment}
\newcommand{\be} {\begin{eqnarray*}}
\newcommand{\ee} {\end{eqnarray*}}
\newcommand{\bbeta} {\boldsymbol{\beta}}
\newcommand{\bSigma} {\boldsymbol{\Sigma}}

\newcommand{\bDelta} {\mbox{\boldmath $ \Delta$}}
\newcommand{\argmax}{\operatornamewithlimits{argmax}}

\newcommand{\bfV} {{\bf V}}

\newcommand{\bb} {{\bf b}}

\newcommand{\argmin}{\mathop{\rm argmin~}}

\newcommand{\bfH}{{\bf H}}
\newcommand{\bfh}{{\bf h}}
\newcommand{\bfg}{{\bf g}}

\newcommand{\bfX}{{\bf X}}
\newcommand{\bfZ} {{\bf Z}}
\newcommand{\bfx}{{\bf x}}
\newcommand{\bfA}{{\bf A}}
\newcommand{\bfa}{{\bf a}}
\newcommand{\bfv}{{\bf v}}
\newcommand{\bfP}{{\bf P}}

\newcommand{\bfq}{{\bf q}}
\newcommand{\bfQ}{{\bf Q}}

\newcommand{\bfu}{{\bf u}}

\newcommand{\bxi} {\mbox{\boldmath $\xi$}}

\newcommand{\X}{\bold{X}}
\newcommand{\V}{\bold{V}}
\newcommand{\1}{\bold{1}}

\newcommand{\vbeta}{\boldsymbol{\beta}}

\def\X{{ \widetilde{X} }}

\newtheorem{theorem}{Theorem}[section]
\newtheorem{lemma}[theorem]{Lemma}

\newtheorem{defn}[theorem]{Definition}
\newtheorem{proposition}[theorem]{Proposition}
\newtheorem{corollary}[theorem]{Corollary}
\newtheorem{rem}{Remark}[section]

\small
\usepackage{natbib}
\usepackage[colorlinks,bookmarksopen,bookmarksnumbered,citecolor=blue,urlcolor=blue,linkcolor=blue]{hyperref}

\usepackage{subcaption}

\begin{document}
\footnotesize
\title{A Modern Theory for High-dimensional Cox Regression Models}
\author[1]{Hanxuan Ye
}
\author[1]{Xianyang Zhang}
\author[2]{Huijuan Zhou}
\date{}
\affil[1]{Texas A\&M University}
\affil[2]{Shanghai University of Finance and Economics}
\maketitle

\textbf{Abstract:} The proportional hazards model has been extensively used in many fields such as biomedicine to estimate and perform statistical significance testing on the effects of covariates influencing the survival time of patients. The classical theory of maximum partial-likelihood estimation (MPLE) is used by most software packages to produce inference, e.g., the \texttt{coxph} function in R and the \texttt{PHREG} procedure in SAS. In this paper, we investigate the asymptotic behavior of the MPLE in the regime in which the number of parameters $p$ is of the same order as the number of samples $n$. The main results are (i) existence of the MPLE undergoes a sharp `phase transition'; (ii) the classical MPLE theory leads to invalid inference in the high-dimensional regime. We show that the asymptotic behavior of the MPLE is governed by a new asymptotic theory. These findings are further corroborated through numerical studies. The main technical tool in our proofs is the Convex Gaussian Min-max Theorem (CGMT), which has not been previously used in the analysis of partial likelihood. Our results thus extend the scope of CGMT and shed new light on the use of CGMT for examining the existence of MPLE and non-separable objective functions. 
\\
\strut \textbf{Keywords:} 
Convex Gaussian Min-max Theorem, Cox Regression, High-dimensionality, Likelihood-ratio Test, Wald Test.

\section{Introduction}
\subsection{Background}
Since the first introduction in 1972 by D. R. Cox, the proportional hazards model has been routinely used in many applied fields such as biomedicine in order to investigate the association between the survival time of patients and predictor variables. In the proportional hazards model, the hazard for an individual, $i$, with covariates $\bfX_i=(X_{i1},\dots,X_{ip})^\top$ is specified as a product $$\lambda(t|\bfX_i)=\lambda_0(t)\exp(\bfX_i^\top \bbeta^*),$$
of an unknown baseline hazard function $\lambda_0(\cdot)$ and a relative risk function $\exp(\bfX_i^\top \bbeta^*)$ in which the individual covariate values enter linearly via the regression coefficients $\bbeta^*=(\beta_1^*,\dots,\beta_p^*)^\top.$ When no prior
knowledge is available regarding the structure of the parameters, the proportional hazards model is often fitted via maximizing the partial log-likelihood function. Classical theory of the maximum partial likelihood estimation (MPLE) states that when the dimension of variables $p$ is fixed and the sample size $n\rightarrow+\infty$,
$$\sqrt{n}(\widehat{\bbeta}-\bbeta^*)\overset{d}{\to}N(0,\mathbf{I}^{-1}_{\bbeta^*}),$$
where $\widehat{\bbeta}$ denotes the maximum partial likelihood estimator, $\mathbf{I}_{\bbeta^*}$ is the $p\times p$ Fisher information matrix evaluated at the true value $\bbeta^*$ and $\overset{d}{\to}$ stands for convergence in distribution. This result has been adopted by many software packages to produce significance testing and confidence intervals e.g., the \texttt{coxph} function in R and the \texttt{PHREG} procedure in SAS.

\subsection{Motivation}
In modern clinical studies, it is often of interest to understand the association between patients’ survival times and a set of high-dimensional covariates such as genomics features and medical images. The use of proportional hazards model to large data sets thus raises the following questions:
\begin{enumerate}
    \item[(A)] does the classical theory of MPLE provide a good approximation to the finite sample behaviors when the number of variables $p$ is a non-negligible proportion of the sample size $n$? 
    \item[(B)] If the classical theory fails in the high-dimension paradigm, is there a new theory characterizing the asymptotic properties of MPLE?  
\end{enumerate}

\subsection{Prior works and our contribution}
Previous works in the survival analysis literature have focused on the sparse regime where the number of relevant predictors is much smaller than the sample size, and employed the penalized partial likelihood approach to perform simultaneous estimation and variable selection \citep{TibshiraniCox1997,FanCox2002,gui2005,zhang2007,BradicCox2011}. Oracle inequalities for the penalized MPLE have been obtained in \cite{Gaiffas2012,Huang2013,Kong2014}. A more recent line of research studies hypothesis testing and confidence interval construction for high-dimensional Cox regression using the debiasing approach \citep{Fang2017,Yu2018,Kong2021}.

In this work, with the aim to answer questions (A) and (B), we study the original MPLE in the high-dimensional setting where $p$ and $n$ diverge to infinity simultaneously with $p/n\to\delta\in (0,1)$. To the best of our knowledge, the asymptotic properties of the original MPLE have not been studied under this asymptotic regime in the literature. Our main results are summarized as follows.
\begin{enumerate}
    \item[(i)] Under the Gaussian assumption on the covariates, the existence of the MPLE undergoes a sharp `phase transition'. The MPLE exists asymptotically (with probability approaching one) only when $\delta$ is below a quantity $h(\lambda_0,\kappa,P_\mathcal{C})$ that is determined by the base line hazard function $\lambda_0$, the signal
    strength $\kappa^2:=\lim \text{var}(\bfX_i^\top\bbeta)$ and the distribution of the censoring time $P_\mathcal{C}$.
    \item[(ii)] The classical MPLE theory leads to invalid inference in the high-dimensional regime where $p/n\to\delta\in (0,1)$. We show that the asymptotic behaviors of the MPLE and the Wald test formed by the sum of squares of the MPLE are governed by a new asymptotic theory. In particular, the asymptotic bias and variance of the MPLE are precisely characterized by the new theory. The Wald test is shown to converge to a scaled chi-square distribution. 
\end{enumerate}

\subsection{Technical tools}
There have been several recent works on understanding the asymptotic behaviors of statistical estimators derived from minimizing a convex loss function in the high-dimensional setting. Examples include the regularized linear regression \citep{ThrampoulidisRegularized2015}, M-estimation \citep{ElKarouiRobust2013,DonohoHigh2016}, penalized M-estimation \citep{ThrampoulidisPrecise2018}, logistic regression \citep{SurModern2019}, reguralized logistic regression \citep{SalehiImpact2019}, high-dimensional classification \citep{Liang2020,t2020}, adversarial training \citep{jm2020} among others. All the above results are derived under the assumptions that $p/n\rightarrow \delta>0$ and most of the works assumed that the covariates follow a Gaussian distribution.
The technical tools employed in these studies can be roughly classified into three categories: (a) the leave-one-out argument; (b) the approximate message passing (AMP) algorithm and the associated state evolution equations; (c) the Convex Gaussian Min-max Theorem (CGMT). In \citet{ElKarouiRobust2013}, the authors developed the leave-one-out technique to heuristically derive a nonlinear system of two deterministic equations that characterizes the asymptotic square errors of the M-estimator. 
A rigorous proof of these results based on the leave-one-out argument was provided in \citet{ElKarouiAsymptotic2013}. The AMP algorithm was first introduced in \citet{Donoho2009} as an efficient reconstruction scheme in compressed sensing. The authors further derived a system of state evolution equations to accurately predict the dynamical behavior of several observables involved in the AMP algorithm. The AMP technique was later on adopted by \citet{DonohoHigh2016} and \citet{SurModern2019} to study the high-dimensional M-estimation and logistic regression respectively.
Along a different line, \citet{ThrampoulidisRegularized2015,ThrampoulidisPrecise2018} introduced the CGMT as a stronger version
of the classical Gaussian Min-max Theorem due to \citet{Gordon1988}. The usefulness of the CGMT lies on that it associates the original primary optimization (PO) with an auxiliary optimization (AO) problem from which one can infer the asymptotic properties regarding the original PO. In many applications, the AO problem can be reduced to an optimization problem involving only scalar variables. The Karush-Kuhn-Tucker conditions with respect to the scalar variables in the AO problem induce a set of equations 
that characterizes the asymptotic properties of the optimal solution to the PO. The CGMT has proved useful  
in several contexts arising from high-dimensional statistics, machine learning and information theory, see e.g., \citet{Dhifallah2018,SalehiImpact2019,Hu2019,Liang2020,t2020,jm2020}. 

The main results (i) and (ii) in this paper are also built upon the CGMT. To obtain (i), we observe that the existence of MPLE is related to the optimal value of a convex optimization problem. Using the CGMT and some results from convex geometry, we prove the phase transition phenomenon for the existence of MPLE and obtain the corresponding phase transition curve. Result (ii) are derived using the CGMT by relating the MPLE to the solution of an AO problem. However, due to the non-separability of the partial likelihood function, our analysis is more involved than those for M-estimation and logistic regression, and extra effort is needed to deal with the AO problem and derive the optimality conditions, see Section \ref{sec:M}. Finally, we emphasize that our arguments are different from those in \citet{SurModern2019} which is built on the AMP technique that does not seem directly applicable to our setting.

The rest of the paper is organized as follows. Section \ref{sec:prelim} introduces the setups and discusses the failures of the classical large sample theories for MPLE in high-dimension. We study the existence of MPLE and derive the phase transition curve in Section \ref{sec:exist}. We develop a new asymptotic theory in Section \ref{sec:asym}, which is used to perform asymptotic exact error analysis on the MPLE and to derive the asymptotic distributions of the MPLE.
We further present some numerical results to corroborate our theoretical findings within each section. Section \ref{sec:future} concludes and discusses a few future research directions.

\section{Preliminaries}\label{sec:prelim}
\subsection{Basic setup}
Consider a sequence of i.i.d samples $\{(\bfX_i,T_i)\}^{n}_{i=1}$ generated from the population $(\bfX,T)$, where $\bfX_i=(X_{i1},\dots,X_{ip})^\top$ is a $p$-dimensional covariate associated with the $i$th individual. In practice, not all the survival times are fully observable. We consider a sequence of right censoring times $\{C_i\}^{n}_{i=1}$ that are independent of the survival times $\{T_i\}^{n}_{i=1}$ (see Remark \ref{rm-c} for a relaxation of this assumption). Thus we work with the i.i.d. observations $(Y_i,
\bfX_i,\Delta_i)$, where $Y_i=T_i\land C_i:=\min(T_i,C_i)$ and $\Delta_i= \mathbf{1}\{T_i\le C_i\}$ are event time and censoring indicator, respectively. The Cox proportional hazards model specifies the hazard function for the $i$th individual as
\begin{align}\label{eq-lambda}
\lambda(t|\bfX_i)=\lambda_0(t)\exp(\bfX_i^\top \bbeta^*),
\end{align}
where $\bbeta^*\in \mathbb{R}^p$ is the parameter of interest and $\lambda_0(t)$ is the unknown baseline hazard function. The maximum partial likelihood estimator (MPLE) is defined as
\begin{align}\label{PL}
\widehat{\bbeta}=\argmax_{\bbeta} L(\bbeta),\quad
L(\bbeta)=&\frac{1}{n}\sum_{i=1}^n\left\{\bfX_i^\top \bbeta-\log\left(\frac{1}{n}\sum_{j=1}^n\mathbf{1}\{Y_j\ge Y_i\}\exp(\bfX_j^\top \bbeta)\right)\right\}\Delta_i,
\end{align}
where $L(\bbeta)$ is the log partial likelihood function evaluated at $\bbeta.$ Compared to M-estimation and logistic regression, the log partial likelihood is a sum of non-i.i.d random variables which complicates the analysis.

\subsection{Failures of classical large sample theories}
In classical large sample theories, we assume $p$ is fixed and let $n\to\infty$. Under mild regularity conditions, the MPLE behaves similarly as the ordinary MLE \citep{MurphyProfile2000}
\begin{align*}
	\sqrt{n}(\widehat{\bbeta}-\bbeta^*)\overset{d}{\to}N(0,\mathbf{I}^{-1}_{\bbeta^*}),
\end{align*}
where $\mathbf{I}_{\bbeta^*}=-\mathbb{E}[\partial^2 L(\bbeta)/\partial \bbeta\partial\bbeta^\top|_{\bbeta=\bbeta^*}]$ is the $p\times p$ Fisher information matrix evaluated at the truth. However, in the comparable setting where $p$ goes to infinity with the same rate as $n$, the classical theories can lead to invalid inference. We use numerical examples to illustrate this point. Through the numerical studies below, we set $n=4,000$ and $p=800$ (so that $\delta=0.2$). Suppose the entries of the design matrix $(\bfX_1,\dots,\bfX_n)^\top$ follow $N(0,1/p)$ independently. We set $\lambda_0(t)=\lambda=1$ for the baseline hazard function and $C_i\sim^{\text{i.i.d.}} \text{Unif}(1,2)$ for the censoring times. We have the following observations which are in general similar to those in \citet{SurModern2019} for high-dimensional logistic regression.

\begin{figure}
	\begin{subfigure}[b]{0.485\textwidth}
		\centering
		\includegraphics[width=\textwidth]{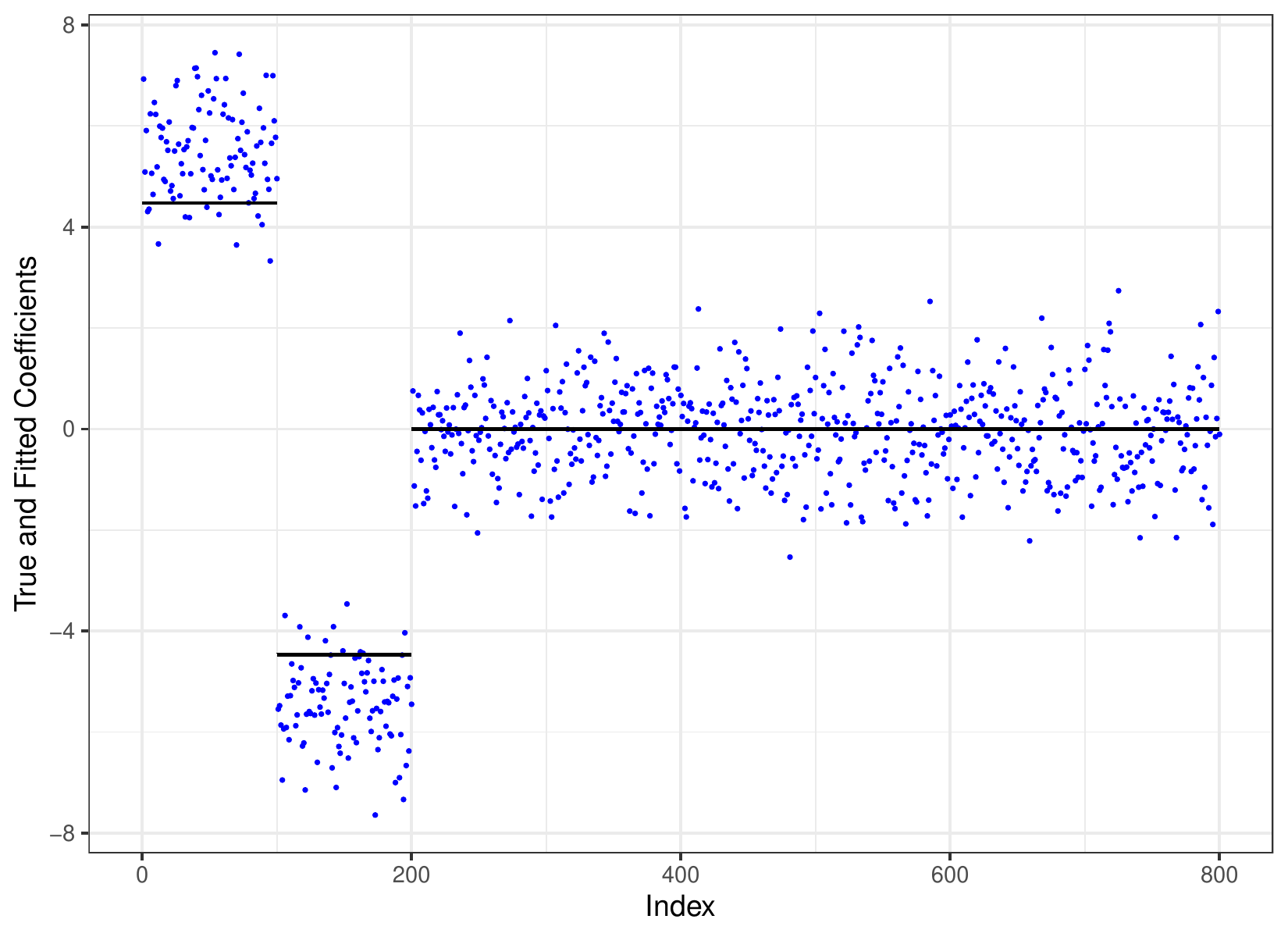}
		\subcaption*{(a) The true and estimated values of the regression coefficients. The dark line segments represent the values of $\bbeta^*$ and blue points represent the values of $\widehat{\bbeta}$ for the corresponding coordinates. }
	\end{subfigure}
	\hfill
	\begin{subfigure}[b]{0.485\textwidth}
		\centering
		\includegraphics[width=\textwidth]{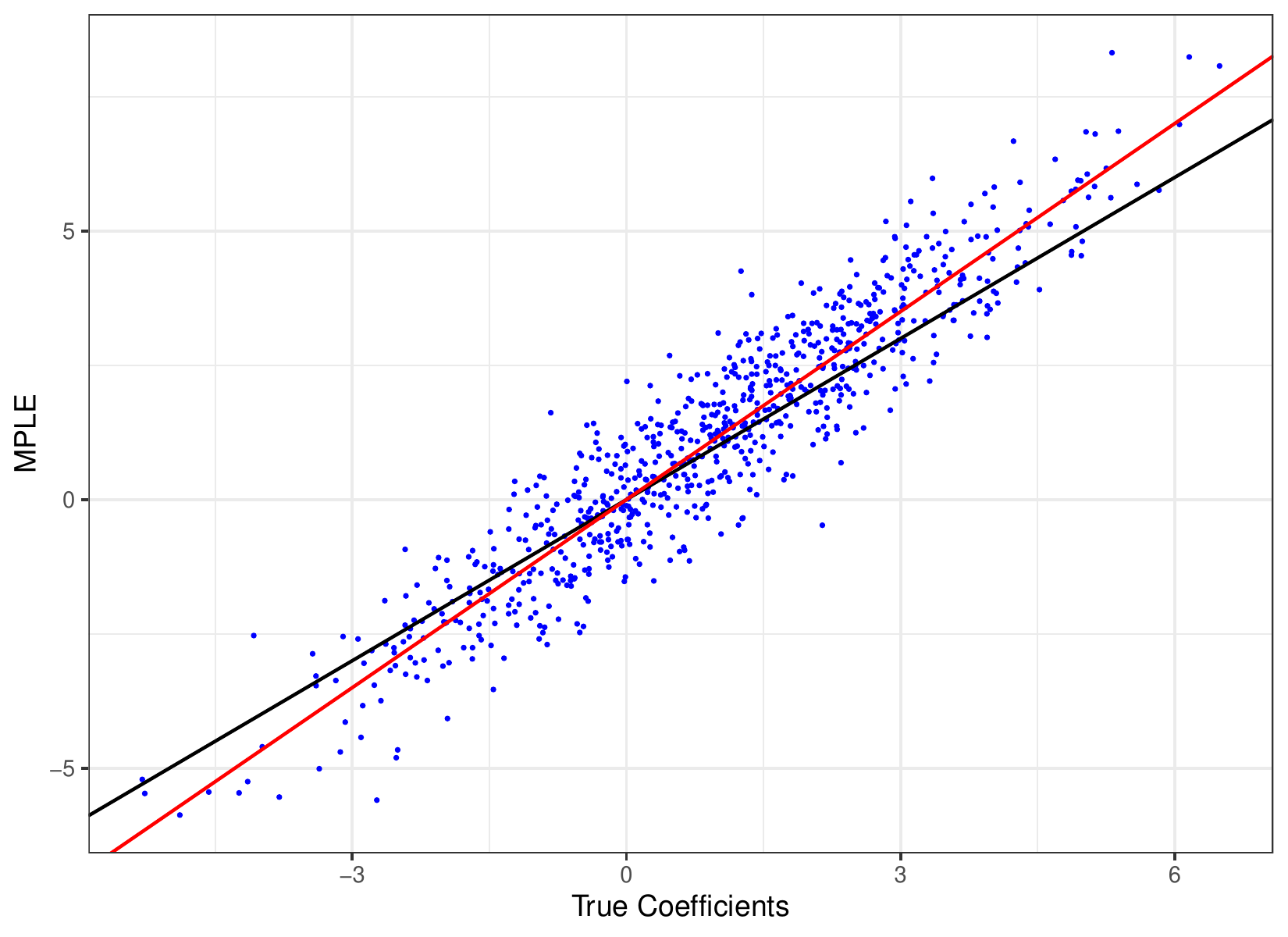}
			\subcaption*{(b) The blue points correspond to the pairs $(\beta_j^*,\widehat{\beta}_j)$ for $j=1,2,\dots,p$. The dark line has slope one, and the red line is the fitted least squares regression line based on the blue points.}
	\end{subfigure}
	\caption{The biasness of the MPLE.}
	\label{fig_betaest}
\end{figure}

\begin{enumerate}
    \item MPLE is biased. In the first experiment, we set the first hundred entries of $\bbeta^*$ to be $2\sqrt{5}$, the next hundred entries to be $-2\sqrt{5}$ and the remaining entries to be 0. It can be clearly seen from Figure \ref{fig_betaest}(a) that MPLE is not unbiased. The absolute values of the estimates tend to be larger than the true values. In the second experiment, we generate the entries of $\bbeta^*$ from $N(1, 4)$ independently. Figure \ref{fig_betaest}(b) shows that the pairs of $(\beta_j^*,\hat{\beta}_j)$ do not scatter around the 45 degree line but rather a different line with a larger slope, which indicates an upward bias in the estimation.

    \item The standard deviation (std.) of $\widehat{\bbeta}$ from the Fisher information matrix (abbreviated as Fisher std.) is smaller than the true std. To see this, we generate half of the entries of $\bbeta^*$ independently from $N(3, 1)$ and let the remaining be zeros. We conduct 1,000 simulation runs, estimate the Fisher std. by the square root of the average of 1,000 diagonals of the inverse of the matrix $-\partial^2 L(\bbeta)/\partial \bbeta\partial\bbeta^\top|_{\bbeta=\bbeta^*}$, and estimate the true std. by the std. of 1,000 estimates of $\bbeta^*$. Figure (\ref{fig_betaest3std_cox}) shows the mean of the 400 estimates of the Fisher stds of the null coefficients and the histogram of the estimates of the true stds of the null coefficients. Apparently, the Fisher std. underestimates the true std.

    \item The partial log-likelihood ratio test does not converge to a chi-square distribution, and the Wald z-test does not converge to a standard normal distribution. Again we let half of the entries of $\bbeta^*$ be generated independently from $N(3, 1)$ and the rest be zeros. We use the partial log-likelihood ratio test to examine the significance of the first null coefficient (i.e., the 401 entry of the coefficient vector). According to the classical large sample theory, the partial log-likelihood ratio test converges in distribution to $\chi^2_1$ \citep{Wilks1938}.
    We conduct 50,000 simulation runs, and calculate the p-values based on the $\chi^2$ approximation. From Figure \ref{fig_pvalues_cox}(a), we see that the distribution of the p-values deviates significantly from the uniform distribution.
    Using the outputs from the previous simulation (for the second bullet point), we can calculate the Wald z-statistics by the ratio between the 1,000 estimates of the 400 null coefficients and their Fisher stds, and then obtain the p-values, see Figure \ref{fig_pvalues_cox}(b). Again the p-values are not uniformly distributed in this case.
\end{enumerate}

\begin{figure}
	\centering
	\includegraphics[scale=0.45]{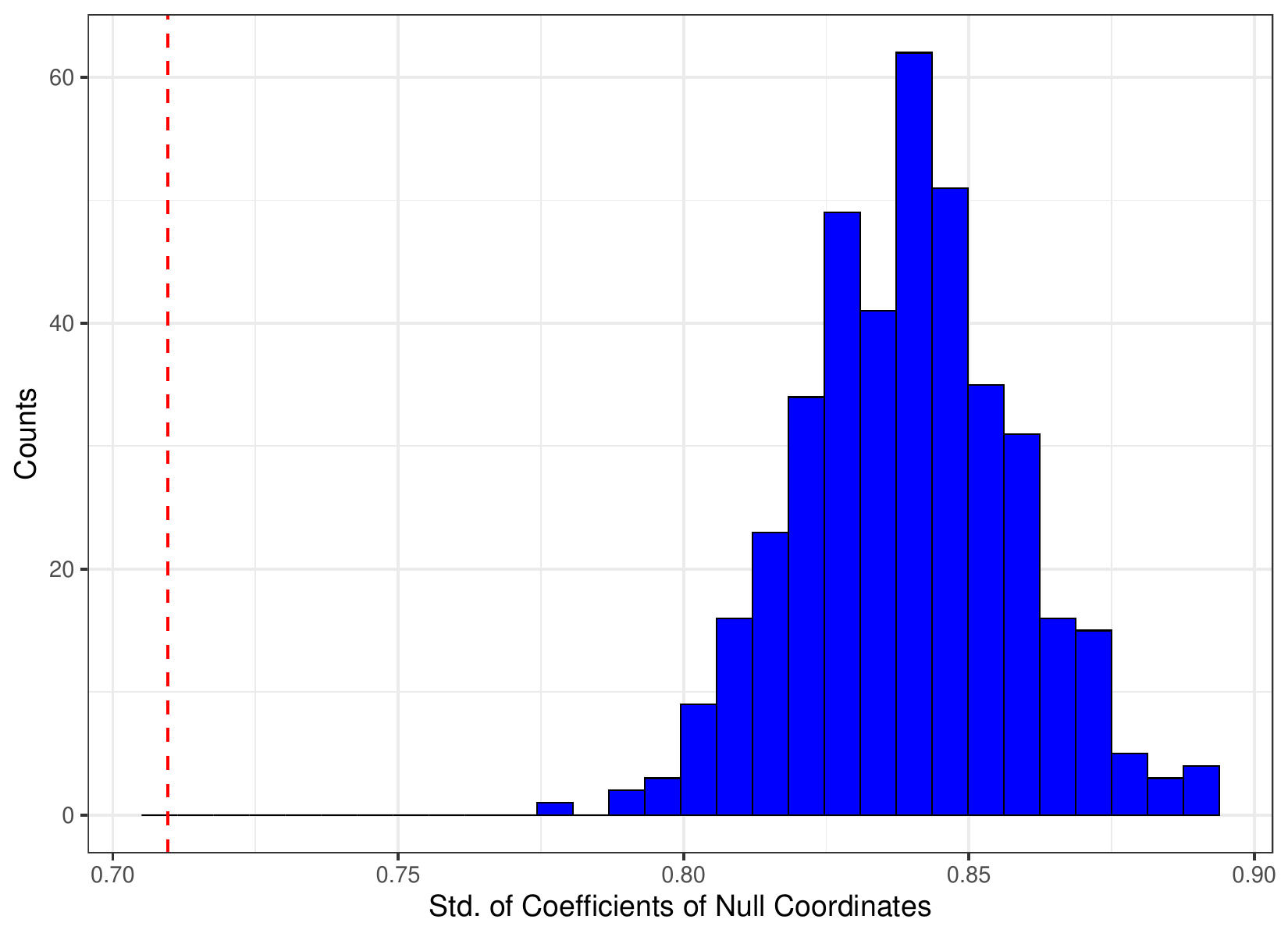}
	\caption{Comparison between the Fisher std. and true std. The red line represents the Fisher std. The blue histogram depicts the empirical distribution 
	of the std's of the 400 null coefficients. 
	}
	\label{fig_betaest3std_cox}
\end{figure}

\begin{figure}
	\begin{subfigure}[b]{0.485\textwidth}
		\centering
		\includegraphics[width=\textwidth]{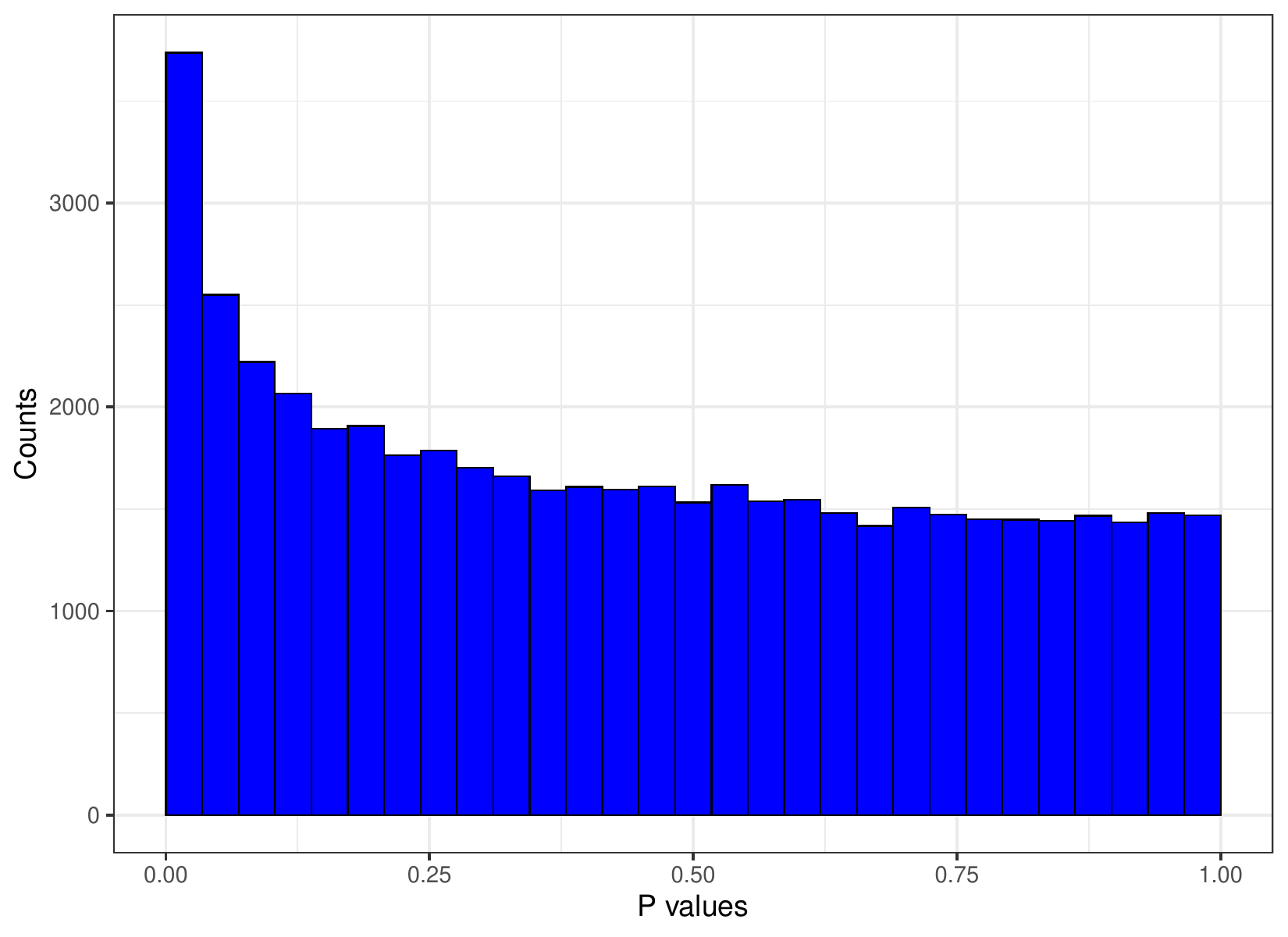}
		\subcaption*{(a) P-values of the partial log-likelihood ratio tests for the null coefficients using the $\chi^2_1$ approximation. }
	\end{subfigure}
	\hfill
	\begin{subfigure}[b]{0.485\textwidth}
		\centering
		\includegraphics[width=\textwidth]{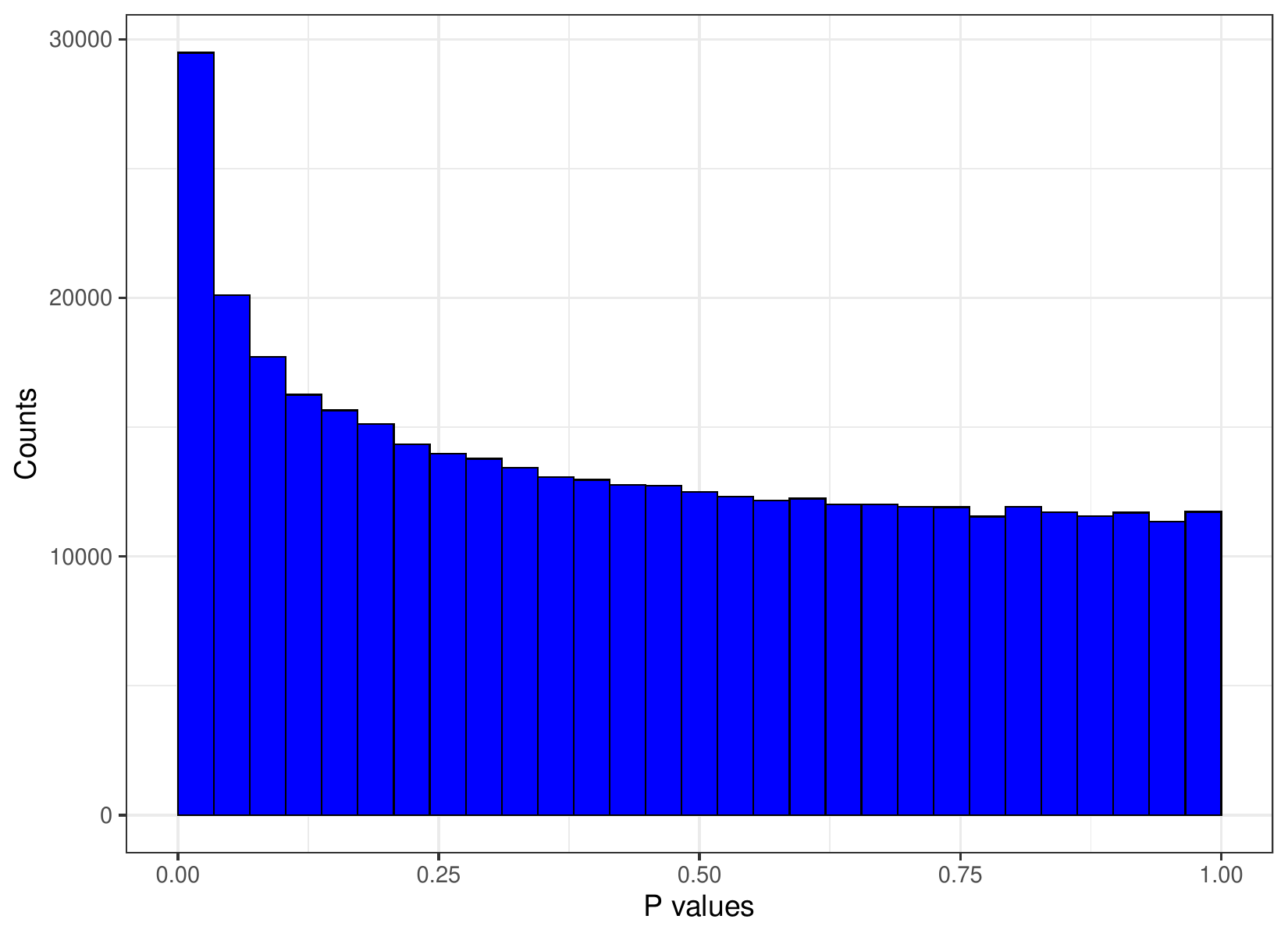}
			\subcaption*{(b) P-values of the Wald-z tests for the null coefficients using the standard normal approximation.}
	\end{subfigure}
	\caption{Invalid inferences based on the classical theory for MPLE.}
	\label{fig_pvalues_cox}
\end{figure}

\section{Existence of the MPLE}\label{sec:exist}
\subsection{Phase transition boundary curve}
As the first step toward understanding the behaviors of the MPLE in high-dimension, we characterize the conditions for the existence of the MPLE. For high-dimensional logistic regression, \citet{CandesPhase2018} established that the existence of the MLE undergoes a phase transition phenomenon, and obtained the explicit form of the boundary curve. However, their argument is not directly applicable to the Cox regression model due to the more complicated characterization of the existence of the MPLE and the model structure. To overcome the difficulty, we present a new argument based on the CGMT technique. The basic idea is to relate the existence of the MPLE to the optimal value of a convex optimization problem (the PO problem). Using the CGMT, we can associate the PO problem with an AO problem. By analyzing the corresponding AO problem, we find the condition under which the MPLE exists with probability approaching one.
Using similar arguments, we manage to recover some of the results in \citet{CandesPhase2018}. The readers are referred to Section \ref{sec:MLE-logistic} for the details.

Throughout the section, we shall assume that $\bfX_i\sim^{\text{i.i.d}} N(0,\bSigma)$ for a non-singular covariance matrix $\bSigma$. 
We first present the general conditions for the existence of the MPLE. Define the set
\begin{align*}
\mathcal{B}:=\text{span}\left\{ \Delta_i(\bfX_j-\bfX_i): 1\leq i\leq n, j\in \mathcal{R}(Y_i)\setminus\{i\}\right\},    
\end{align*}
where $\mathcal{R}(t)=\{j:Y_j\geq t\}$. By \citet{JacobsenExistence1989}, the MPLE exists if and only if the following two conditions are satisfied: 
\begin{enumerate}
    \item $\text{dim}(\mathcal{B})=p$;
    \item There does not exist a nonzero vector $\bb\in\mathbb{R}^p$ such that 
    $$\bb^\top (\bfX_j-\bfX_i)\leq 0,$$
    for all $1\leq i\leq n$ with $\Delta_i=1$ and $j\in\mathcal{R}(Y_i)\setminus \{i\}$.
\end{enumerate}
Suppose $C_i\geq c_L$ and $P(T_i<c_L)>c>0.$ Then with probability tending to one, there exists a $Y_i$ with $Y_i<c_L$ and $\Delta_i=1$. In this case, Condition 1 holds with probability approaching one. By writing $\bfX_i=\bSigma^{1/2}\bfZ_i$ for $\bfZ_i\sim^{\text{i.i.d}} N(0,\mathbf{I}_p)$, Condition 2 can be equivalently expressed as: there does not exist a nonzero vector $\bb\in\mathbb{R}^p$ such that 
$\bb^\top (\bfZ_j-\bfZ_i)\leq 0,$
for all $1\leq i\leq n$ with $\Delta_i=1$ and $j\in\mathcal{R}(Y_i)\setminus \{i\}$.
Therefore, without loss of generality, we may assume that $\bSigma=\mathbf{I}_p$ in the following discussions. Define the set 
\begin{align*}
\mathcal{D}:=\{(i,j): 1\leq i\leq n, \Delta_i=1, j\in\mathcal{R}(Y_i)\setminus \{i\}\}, 
\end{align*}
and let 
$$\kappa^2=\text{var}(\bfX_i^\top\bbeta^*).$$ 
By the rotational invariance of the Gaussian distribution, we can show that the joint distribution of $(Y_i,\bfX^\top_i)=(Y_i,X_{i1},\dots,X_{ip})$ is the same as that of $$(y_i,\bfq_i^\top)=(y_i,q_{i1},\dots,q_{ip}),$$ 
where $y_i=t_i\wedge C_i$ with $t_i$ having the hazard function 
$$\lambda(t|q_{i1})=\lambda_0(t)\exp(\kappa q_{i1})$$ 
and 
\begin{align*}
&\bfq_i=(q_{i1},\dots,q_{ip})^\top\sim N(0,\mathbf{I}_p),\quad (q_{i2},\dots,q_{ip})\perp(y_i,q_{i1}),
\end{align*}
for $1\leq i\leq n.$ To examine the existence of the MPLE, we consider the convex optimization problem
\begin{equation}\label{eq-linear}
\begin{split}
&\max_{-1\leq \bb\leq 1}\sum_{(i,j)\in\mathcal{D}}a_{ij}\bb^\top(\bfq_i-\bfq_j)\\ 
&\text{subject to} \quad \bb^\top(\bfq_i-\bfq_j)\geq 0 \text{ for all }(i,j)\in\mathcal{D},
\end{split}
\end{equation}
where $a_{ij}>0$ is prespecified and fixed, $\bb=(b_1,\dots,b_p)^\top$ and $-1\leq \bb\leq 1$ means $-1\leq b_i\leq 1$ for all $i$. Clearly, the MPLE does not exist if and only if the optimal value of the above problem is greater than zero. Before presenting the main result regarding the existence of the MPLE, we introduce some quantities. 
Without loss of generality, we assume that
$$y_1\geq y_2\geq \cdots \geq y_n,$$
and the indices of censored observations is smaller than the indices of uncensored observations that have the same value. Let $\{2\leq i\leq  n:\Delta_i=1\}=\{i_1,\dots,i_k\}$. Define the set
\begin{align*}
\mathcal{M}=\left\{\mathbf{m}=(m_1,\dots,m_n)\in\mathbb{R}^n: \min_{s<i_l}m_{s}\geq m_{i_l}, \max_{j\in D_{i_l}}m_j\leq m_{i_l}, l=1,\dots,k\right\},
\end{align*}
where $D_{i_l}=\{1\leq j<i_l: y_j=y_{i_l},\Delta_j=1\}$.
We are now in position to present the main result of this section.
\begin{theorem}\label{thm-1}
Define the quantities
\begin{align}
&h_U(\lambda_0,\kappa,P_\mathcal{C})=\limsup\frac{1}{n}\min_{t\in\mathbb{R},\mathbf{m}\in\mathcal{M}}\|\bfh-t \widetilde{\bfq}-\mathbf{m}\|^2, \label{eq-hu}\\
&h_L(\lambda_0,\kappa,P_\mathcal{C})=\liminf\frac{1}{n}\min_{t\in\mathbb{R},\mathbf{m}\in\mathcal{M}}\|\bfh-t \widetilde{\bfq}-\mathbf{m}\|^2,\label{eq-hl}
\end{align}
where $\widetilde{\bfq}=(q_{11},\dots,q_{n1})^\top$ and $\bfh\sim N(0,\mathbf{I}_n)$ is independent of $\{(y_i,q_{i1})\}^{n}_{i=1}$. The MPLE exists (with probability tending to one) if $\delta<h_L(\lambda_0,\kappa,P_\mathcal{C})$ and the MLE does not exist (with probability tending to one) if
$\delta>h_U(\lambda_0,\kappa,P_\mathcal{C})$. When $h_U(\lambda_0,\kappa,P_\mathcal{C})=h_L(\lambda_0,\kappa,P_\mathcal{C})=h(\lambda_0,\kappa,P_\mathcal{C})$, the MPLE undergoes a phase transition with $h(\lambda_0,\kappa,P_\mathcal{C})$ being the boundary curve.
\end{theorem}
\begin{rem}\label{rm-q}
{\rm 
The restriction in $\mathcal{M}$ can be equivalently expressed as the following pairwise constraints
\begin{align*}
\begin{cases}
m_i\leq m_j \quad &\Delta_i\mathbf{1}\{y_j\geq  y_i\}=1,\Delta_j\mathbf{1}\{y_i\geq y_j\}=0,\\
m_i\geq m_j \quad &\Delta_i\mathbf{1}\{y_j\geq  y_i\}=0,\Delta_j\mathbf{1}\{y_i\geq y_j\}=1,\\
m_i=m_j \quad &\Delta_i\mathbf{1}\{y_j\geq  y_i\}=1,\Delta_j\mathbf{1}\{y_i\geq y_j\}=1,\\
\text{no restriction} \quad &\Delta_i\mathbf{1}\{y_j\geq  y_i\}=0,\Delta_j\mathbf{1}\{y_i\geq y_j\}=0.
\end{cases}    
\end{align*}
Under the assumption that $y_1> y_2>\dots>y_n$ (i.e., there is no tie), we argue that the optimization with respect to $\mathbf{m}\in\mathcal{M}$ in the definitions of $h_U$ and $h_L$ can be translated into a quadratic programming (QP) with at most $n-1$ inequality constraints. For each $1\leq i\leq n-1$, let $k_i$ be the smallest index such that $k_i>i$ and $\Delta_{k_i}=1$. Let $\mathcal{G}$ be the set of indices $i$ such that the corresponding $k_i$ exists. Then for fixed $t\in\mathbb{R}$, the optimization with respect to $\mathbf{m}\in\mathcal{M}$ in (\ref{eq-hu}) and (\ref{eq-hl}) can be formulated as
\begin{align*}
&G_n(t):=\min_{\mathbf{m}=(m_1,\dots,m_n)\in\mathcal{M}} \|\widetilde{\bfh}_t-\mathbf{m}\|^2,\\
&\text{subject to } m_i\geq m_{k_i} \text{ for } i\in\mathcal{G},
\end{align*}
with $\widetilde{\bfh}_t=\bfh-t \widetilde{\bfq}$, which can be solved efficiently using existing QP solvers. By performing an one-dimensional optimization, we can find $\min_{t\in\mathbb{R}} G_n(t)=\min_{t\in\mathbb{R},\mathbf{m}\in\mathcal{M}}\|\bfh-t \widetilde{\bfq}-\mathbf{m}\|^2$. 
}
\end{rem}


\subsection{Checking the existence of MPLE in finite sample}\label{sec:finite-exist}
Next, we discuss how to solve the convex optimization problem (\ref{eq-linear}) by reducing the number of constraints and conduct a numerical study to compare the phase transition boundary curve with the empirical results. 
Indeed we can infer that the number of constraints is no more than $2(n-1)$. More precisely, the number of constraints is equal to $i_k-1+\sum_{l=1}^s(m_l-1)$, where $i_k$ is the maximum index of uncensored observations, $s$ is the number of tie values that have at least two uncensored observations, and $m_l$ is the number of uncensored observations that are equal to the $l$th tie value. In fact, we can write down the constraints explicitly. Let $\{i_1,\dots,i_k\}=\{1\leq i\leq n:\Delta_i=1\}$ with $i_1<\cdots<i_k$. Let $\{j_{l,1},\dots,j_{l,m_l}\}\subseteq\{i_1,\dots,i_k\}$ with $j_{l,1}<\cdots<j_{l,m_l}$ be the index set of the uncensored observations that are equal to the $l$th tie value (with at least two uncensored observations) for $l=1,\dots,s$. Then the full set of constraints are given by
\begin{align}
    &\bb^\top(\bfq_{i_l}-\bfq_{i_{l-1}})\ge0,\bb^\top(\bfq_{i_l}-\bfq_{i_{l-1}+1})\ge0,\dots,\bb^\top(\bfq_{i_l}-\bfq_{i_{l}-1})\ge0,\quad l=1,\dots,k, \label{con-linear-1}\\
    &\bb^\top(\bfq_{j_{l,1}}-\bfq_{j_{l,2}})\ge0,\bb^\top(\bfq_{j_{l,2}}-\bfq_{j_{l,3}})\ge0,\dots,\bb^\top(\bfq_{j_{l,m_l-1}}-\bfq_{j_{l,m_l}})\ge0,\quad l=1,\dots,s,\label{eq-add}
\end{align}
where $i_0=1$ and (\ref{eq-add}) is the additional set of constraints due to the existence of ties. When there is no tie, we only need the $i_k-1$ constraints in (\ref{con-linear-1}). Under the constraints in (\ref{con-linear-1}) and (\ref{eq-add}) and using the simple fact that $\bb^\top(\bfq_i-\bfq_j)\ge 0$ and $\bb^\top(\bfq_j-\bfq_k)\ge 0$ imply that $\bb^\top(\bfq_i-\bfq_k)\ge 0$, one can recover all the constraints in (\ref{eq-linear}). Therefore, the existence of the MPLE can be solved efficiently through the linear programming (\ref{eq-linear}) with  $i_k-1+\sum_{l=1}^s(m_l-1)$ constraints. 


We empirically verify that the existence of the MPLE undergoes a phase transition and the finite sample transition boundary matches well with the theoretical boundary curve derived in Theorem \ref{thm-1}. To generate the data, we assume that each $\beta_i^*/c_\kappa$ is independently generated from the uniform distribution on $[\kappa-1,\kappa+1]$, where the scaling parameter $c_\kappa=\sqrt{ \kappa^2/(1/3 + \kappa^2)}$ ensures that $\|\vbeta^*\|/\sqrt{p}=\kappa$. The survival time $T_i$ follows the exponential distribution with the rate parameter $\lambda_i=\exp(\bfX_i^\top\bbeta^*)$, and the censoring time $C_i$ follows the uniform distribution on $[1,2]$ which is independent of $T_i$. As there is no tie in $\{Y_i\}$, we can examine the existence of the MPLE by solving problem (\ref{eq-linear}) with the $i_m-1$ constraints given in (\ref{con-linear-1}).
Figure~\ref{fig:phase_500} (a) summarizes the results based on $n=500$ and $500$ replications. The red theoretical boundary curve that separates the $\delta-\kappa$ plane into two regions is obtained by solving the constrained quadratic programming described in Remark \ref{rm-q}. While the white and black regions obtained by solving the problem (\ref{eq-linear}) indicate the probability that the MPLE exists (black is zero, and white is one). Overall, the finite sample transition boundary is consistent with the theoretical boundary, which demonstrates the practical relevance of the theoretical finding. In addition, we explore the change of the phase transition boundary with different censoring time distributions. Consider $C_i\sim U[1, b+1]$, where $b \in \{0.5, 1, 7, 30 \}$. The phase transition boundary in Figure~\ref{fig:phase_500} (b) shifts from the right to the left as $b$ decreases, which makes intuitive sense as for a higher censoring rate (i.e., smaller $b$), the existence of the MPLE requires a smaller $\delta$.

\begin{figure}[ht!]
    \centering
    \begin{subfigure}[t]{0.485\textwidth}
        \includegraphics[width=\textwidth]{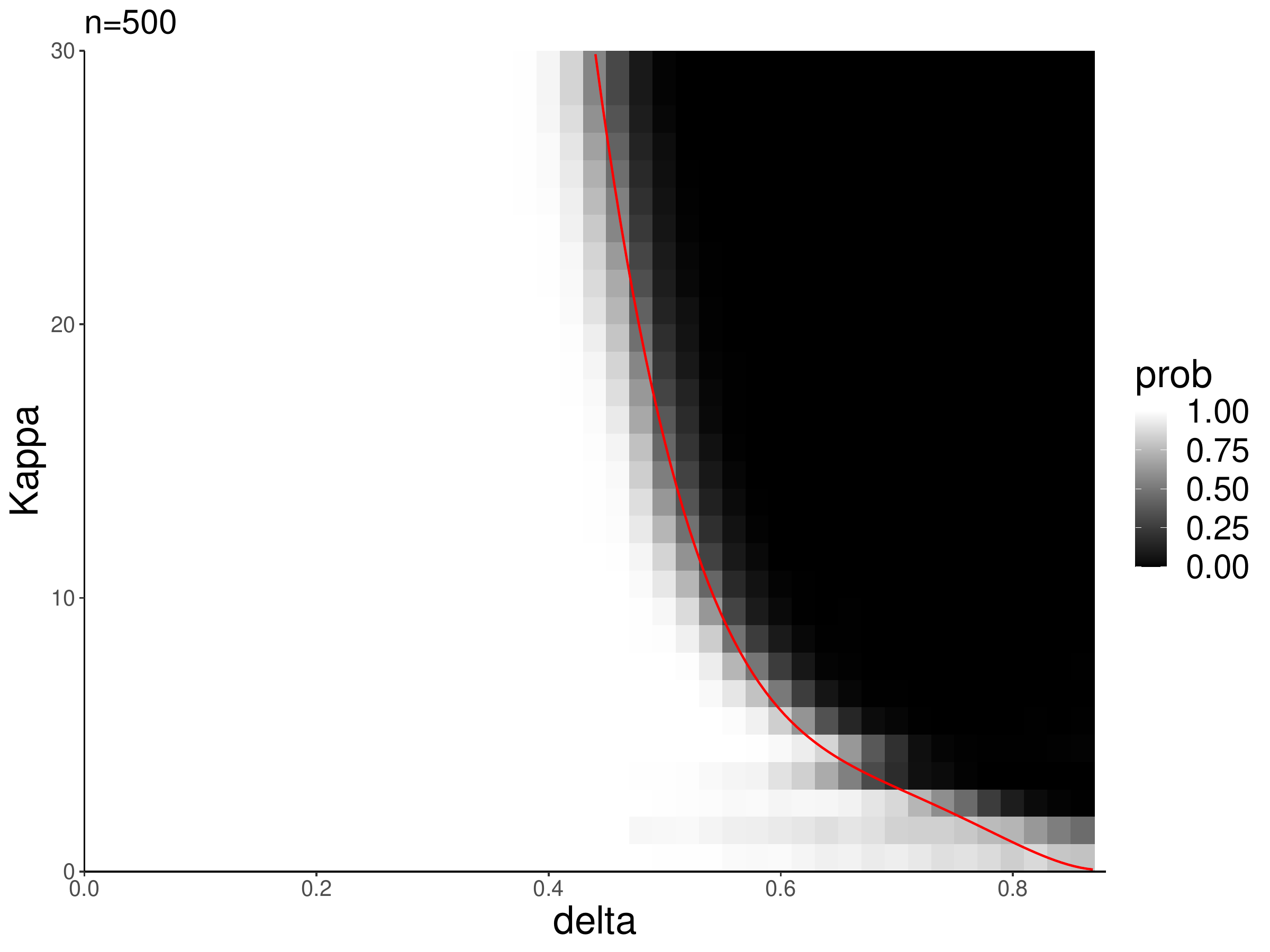}
        \caption{Empirical probability that the MPLE exists (black is zero, and white is one) estimated based on $n=500$ samples and 500 replications. The red curve indicates the theoretical transition boundary.}
    \end{subfigure}
    \hfill
    \begin{subfigure}[t]{0.485\textwidth}
        \includegraphics[width = \textwidth]{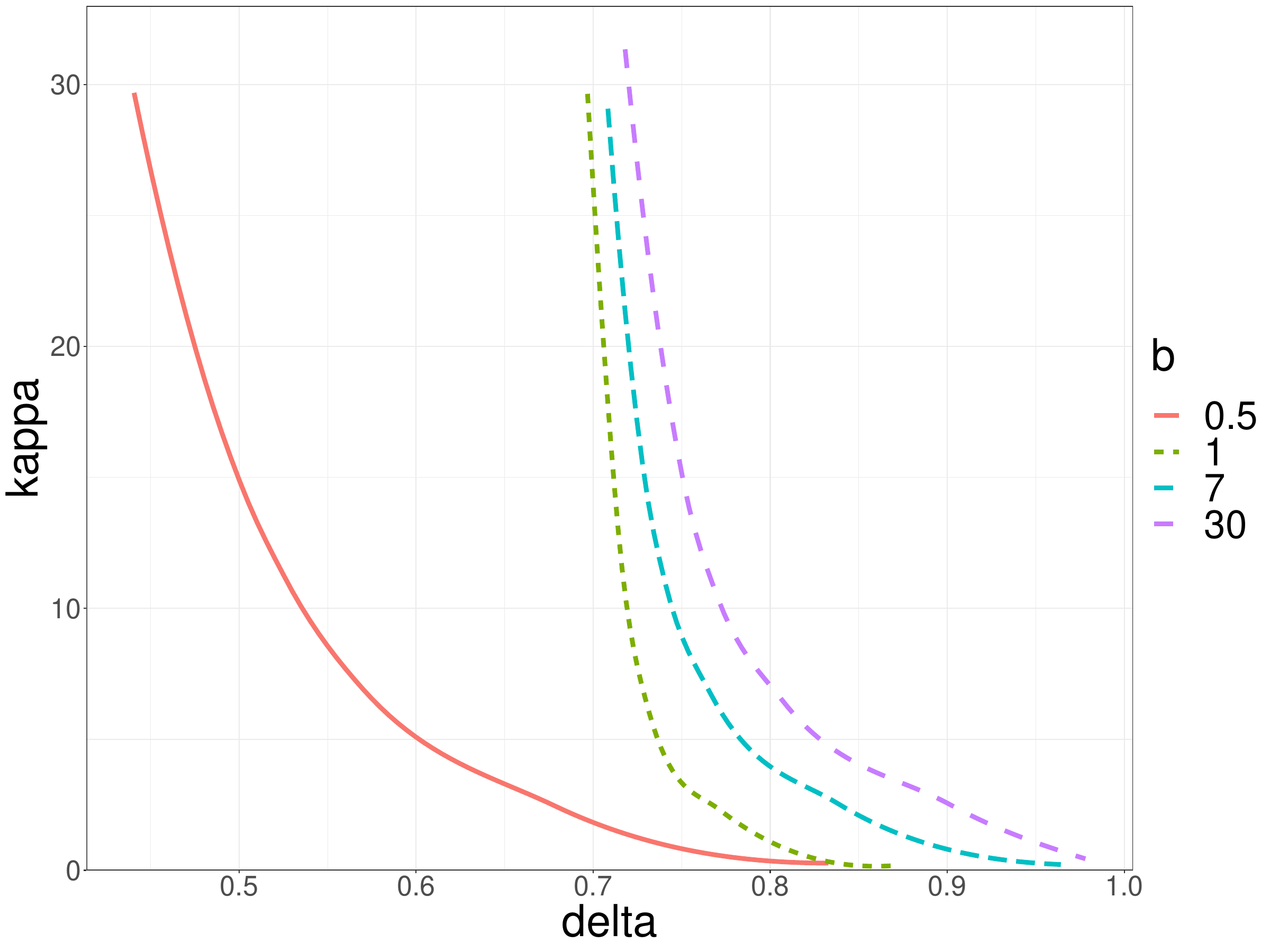}
        \caption{The theoretical transition boundary curves when the censoring time $C_i\sim U[1, b+1]$ with $b\in\{0.5,1,7,30\}$.}
    \end{subfigure}
    \caption{Theoretical transition boundary and the empirical probability that the MPLE exists.}
    \label{fig:phase_500}
\end{figure}



\section{A New Asymptotic Theory}\label{sec:asym}
\subsection{Error analysis}
We develop a new asymptotic theory to describe the asymptotic behavior of the MPLE in the high-dimensional setting. The core of our theory is a set of nonlinear equations derived using CGMT that characterize the behavior of the MPLE. Built upon these equations, we perform an asymptotic exact error analysis on the MPLE and study the asymptotic distributions of the MPLE. Throughout the discussions below, we assume that 
\begin{itemize}
    \item[A1] $X_{ij}\sim^{\text{i.i.d}} N(0,1/p)$ for $1\leq i\leq n$ and $1\leq j\leq p$; 
    \item[A2] $p/n\rightarrow \delta\in(0,1)$.
\end{itemize}
Recall that under the proportional hazards model (\ref{eq-lambda}), the survival function of the survival time $T_i$ is given by
\begin{align*}
S(x|\bfX_i^\top \bbeta^*)=\exp\left\{-\exp(\bfX_i^\top \bbeta^*)\int^{x}_{0}\lambda_0(t)dt\right\},
\end{align*}
As $\kappa^2=\text{var}(\bfX_i^\top\bbeta^*)=\|\beta^*\|^2/p$, $\bfX_i^\top \bbeta^*=^d \kappa Z$ for $Z\sim N(0,1)$, where ``$=^d$'' means equal in distribution. The next assumption can be justified under mild conditions using the law of large numbers.
\begin{itemize}
    \item[A3] Assume that
\begin{align}
&\frac{1}{n} \sum^{n}_{i=1}\mathbf{1}\{T_i\leq C_i\}\rightarrow^p 1-\mathbb{E}[S(C|\kappa Z)],\\
&  \frac{1}{\kappa n}\sum^{n}_{i=1}\mathbf{1}\{T_i\leq C_i\}\bfX_i^\top \bbeta^*\rightarrow^p -\mathbb{E}[S(C|\kappa Z)Z].
\end{align}
where $(C,\kappa Z)=^d (C_i,\bfX_i^\top \bbeta^*)$.
\end{itemize}

Let $a,b$ and $r$ be three scalar quantities that are used to describe the asymptotic behavior of MPLE. The roles of $a$ and $b$ will be made clear later. Further let $\mathbf{q}=(q_1,\dots,q_n)^\top$ with $ q_i=\bfX_i^\top\bbeta^*/\kappa$ and
$\bfh\sim N(0,\mathbf{I}_n)$ that is independent with $(Y_i,
\bfX_i,\Delta_i,C_i)$. Set $\bxi=(\xi_1,\dots,\xi_n)^\top =\kappa a\mathbf{q}+b \bfh+\bDelta \sqrt{\delta} b/r$, where $\bDelta=(\Delta_1,\dots,\Delta_n)^\top.$ We write $G_n(\bfu):=\sum_{i=1}^n\Delta_i\log\left(n^{-1}\sum_{j=1}^n\mathbf{1}\{Y_j\ge Y_i\}\exp(u_j)\right)$ for $\bfu=(u_1,\dots,u_n)$.
Define the proximal operator of $G_n(\bfu)$ at $\bxi$ as
\begin{align*}
&\bfu^*=(u_1^*,\dots,u_n^*)=\arg\min_{\bfu\in\mathbb{R}^n} G_n(\bfu) +\frac{r}{2\sqrt{\delta} b}\left\|\bfu-\bxi\right\|^2.
\end{align*}
To introduce the main result, we require convergence of some counting processes at $\bfu^*$.
Let $Y_i(t)\in \{0,1\}$ be a predictable at risk indicator process which takes the value one when the $i$th subject is under observation \citep{AG82}. We make the following weak convergence assumption.
\begin{itemize}
    \item[A4] There exist processes $S(s,t)$, $S(t)$ and $R(t)$ such that     
    \begin{align*}
    & \frac{1}{n}\sum^{n}_{i=1}Y_i(s)Y_i(t)\exp(2u_i^*)\rightarrow^p S(s,t),    \\
   & \frac{1}{n}\sum^{n}_{i=1}Y_i(t)\exp(u_i^*)\rightarrow^p S(t), \\
  & \frac{1}{n}\sum^{n}_{i=1}Y_i(t)\exp(\bfX_i^\top \bbeta^*)\lambda_0(t)\rightarrow^p R(t).
\end{align*}
\end{itemize}
For $b_1,b_2,b_3\in\mathbb{R}$, consider the nonlinear equation 
$b_3\{\log(u)-b_1\}=
-b_2u$ 
with respect to $u>0$. Let $K(b_1,b_2,b_3)$ be the solution to the equation, i.e., $b_3[\log\{K(b_1,b_2,b_3)\}-b_1]=-b_2K(b_1,b_2,b_3)$. We introduce the following function
\begin{align*}
&M\left(\kappa a,b,\frac{b\sqrt{\delta}}{r}\right) \\  =&\int^{1}_{0}\log(S(s))R(s)ds+\frac{b}{2r\sqrt{\delta}}\int^{1}_{0}\int^{1}_{0}\frac{E\left[Y(s)Y(t)K^2\left(\xi,\int^{1}_{0}\frac{Y(u)}{S(u)}R(u)du,r/(b\sqrt{\delta})\right)\right]}{S(s)S(t)}R(s)R(t)ds dt,  
\end{align*}
where
$S(\cdot)$ is the solution to the equation
\begin{align*}
S(s)=E\left[Y(s)K\left(\xi,\int^{1}_{0}\frac{Y(u)}{S(u)}R(u)du,\frac{r}{b\sqrt{\delta}}\right)\right]
\end{align*}
and
$(Y(t),\xi)=^d (Y_i(t),\xi_i)$. Denote the partial derivative of $M(\cdot,\cdot,\cdot)$ by
\begin{align*}
M_i(a_1,a_2,a_3)=\frac{\partial M(a_1,a_2,a_3)}{\partial a_i},\quad 1\leq i\leq 3.
\end{align*}
\begin{theorem}
Under Assumptions A1-A4, the asymptotic behavior of the MPLE is governed by the following three nonlinear equations:
\begin{align}
&M_1\left(\kappa a,b,\frac{b\sqrt{\delta}}{r}\right)
=-\mathbb{E}[S(C|\kappa Z)Z], \label{neq-1}\\
&M_2\left(\kappa a,b,\frac{b\sqrt{\delta}}{r}\right) =\sqrt{\delta} r,\label{neq-2}\\
&M_3\left(\kappa a,b,\frac{b\sqrt{\delta}}{r}\right)
=-\frac{r^2}{2}
+\frac{1}{2}\left(1-\mathbb{E}[S(C|\kappa Z)]\right).\label{neq-3}
\end{align}
\end{theorem}
Let $(a^*,b^*,r^*)$ be the solution to the nonlinear equations (\ref{neq-1})-(\ref{neq-3}). We have the following result connecting $(a^*,b^*)$ with the asymptotic error of the MPLE.
\begin{theorem}\label{thm-error}
Under Assumptions A1-A4, we have
\begin{align*}
&\frac{\|\widehat{\bbeta}- \bbeta^*\|^2}{\|\bbeta^*\|^2}
\rightarrow^p (a^*-1)^2+\frac{ (b^*)^2}{\kappa^2},\\
&\frac{\|\widehat{\bbeta}-a^*\bbeta^*\|^2}{p}\rightarrow^p (b^*)^2.
\end{align*}
\end{theorem}
The proof of Theorem \ref{thm-error} relies on showing that
\begin{equation}\label{eq-ab}
\widehat{\vbeta}^\top \vbeta^*/\|\vbeta^*\|^2\rightarrow^p a^*,\quad  \|\bfP^\perp\vbeta^*\|/\sqrt{p}\rightarrow^p b^*.   
\end{equation}
In other words, $a^*\|\vbeta^*\|$ measures the projection of the MPLE onto the direction of the true parameter $\vbeta^*$, and $\sqrt{p}b^*$ is approximately the norm of the projection of the MPLE onto the space spanned by the columns of $\bfP^\perp$.
We conduct a numerical study to verify (\ref{eq-ab}) by
following the same data generating mechanism considered in Section \ref{sec:finite-exist}.
Fixing $n=500$, we vary $\delta$ from 0.1 to 0.4 and $\kappa$ from $1$ to 6. Denote by $\hat{a}= \widehat{\vbeta}^\top \vbeta^*/\|\vbeta^*\|^2$ and $\hat{b}= \|\widehat{\vbeta}-\hat{a}\vbeta^*\|/\sqrt{p}$.
We obtain $(a^*,b^*)$ by finding an approximate solution to the nonlinear equations (\ref{neq-1})-(\ref{neq-3}). See Section \ref{sec:error-analysis} in the supplementary material for the details. As seen from Figure~\ref{fig:abplot_500}. $\hat{a}$ and $\hat{b}$ are quite consistent with their theoretical values $a^*$ and $b^*$ in all cases.

\begin{figure}[ht!]
    \centering
    \includegraphics[width = \textwidth]{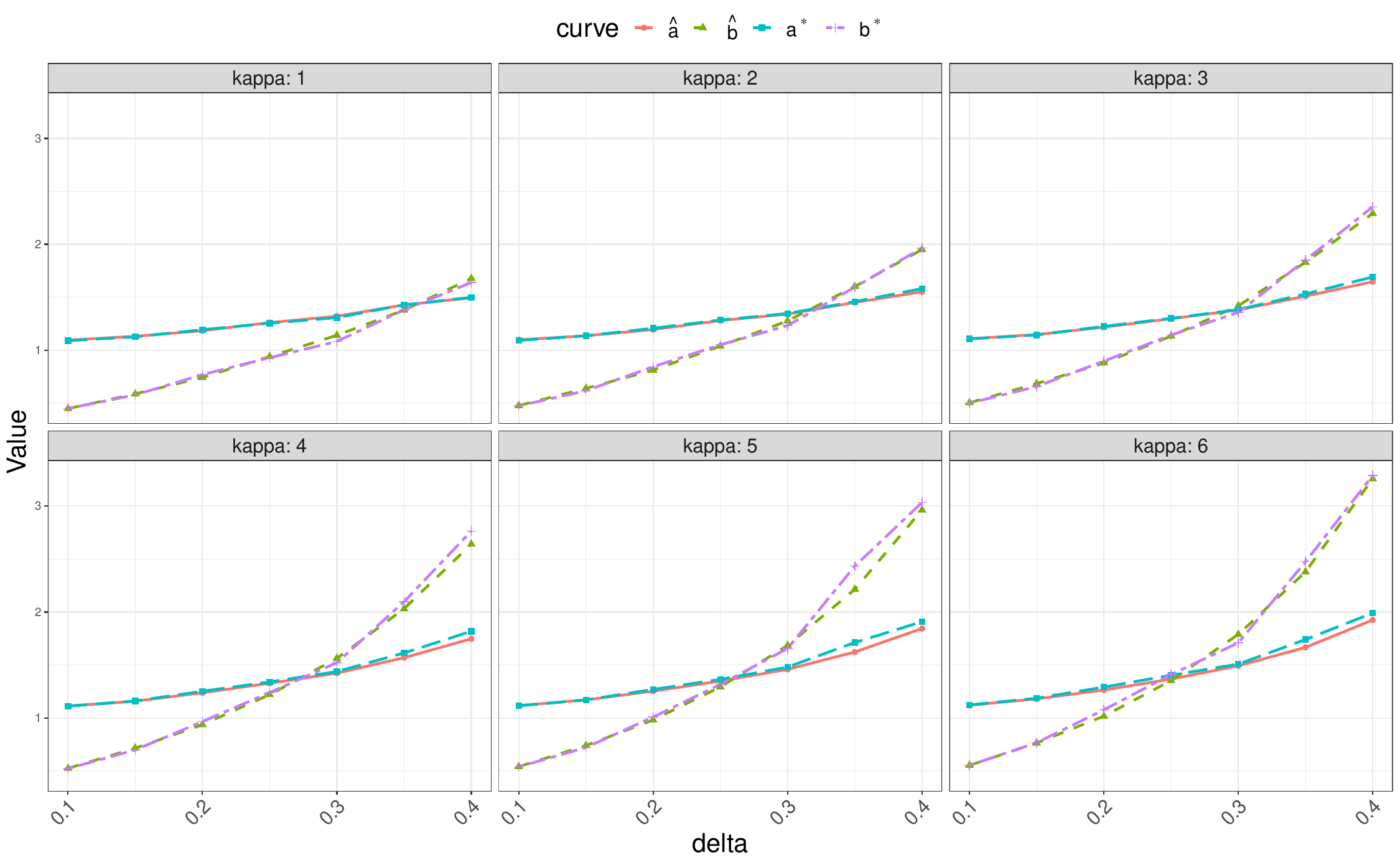}
    \caption{Comparison between $(\hat{a},\hat{b})$ and $(a^*,b^*)$ for various values of $\delta$ and $\kappa$, where $n=500$ and the number of replications is 100.}
    \label{fig:abplot_500}
\end{figure}

\begin{rem}\label{rm-err}
{\rm
Let $\bfg\sim N(0,\mathbf{I}_p)$ be independent of other random quantities. Define $\bfP=\bbeta^*\bbeta^{*\top}/\|\bbeta^*\|^2$ and 
$\bfP^\perp=\mathbf{I}_p-\bfP.$ From the derivations in the analysis of the AO in Section \ref{sec:error-analysis}, we know that 
\begin{align*}
\widehat{\bbeta}=\bfP\widehat{\bbeta}+\bfP^\perp\widehat{\bbeta}
\approx a^*\bbeta^*+ \frac{\bfP^\perp\widehat{\bbeta}}{\|\bfP^\perp\widehat{\bbeta}\|}\|\bfP^\perp\widehat{\bbeta}\|
\approx a^*\bbeta^*+ \frac{\bfP^\perp\bfg}{\|\bfP^\perp\bfg\|}\|\bfP^\perp\widehat{\bbeta}\| 
\approx a^*\bbeta^* + b^*\bfP^\perp\bfg,
\end{align*}
where the first approximation is due to $\widehat{\bbeta}^\top \bbeta^*/\|\bbeta^*\|^2\approx a^*$, the second approximation is because of $\bfP^\perp\widehat{\bbeta}/\|\bfP^\perp\widehat{\bbeta}\|\approx \bfP^\perp\bfg/\|\bfP^\perp\bfg\|$ and the third approximation is from the fact that $\|\bfP^\perp\widehat{\bbeta}\|/\sqrt{p}\approx b^*$ and $\|\bfP^\perp\bfg\|/\sqrt{p}\rightarrow^p 1$. 
Suppose the entries of $\bbeta^*$ are drawn independently from a distribution $P_0$.
For a continuous bivariate function $\psi(\cdot,\cdot)$, we expect that
\begin{align*}
\frac{1}{p}\sum^{p}_{j=1}\psi(\widehat{\beta}_j-a^*\beta^*_j,\beta^*_j)
\approx \frac{1}{p}\sum^{p}_{j=1}\psi((b^*\bfP^\perp\bfg)_j,\beta^*_j)\rightarrow^p E[\psi(b^*Z,\beta_0)],
\end{align*}
where $(b^*\bfP^\perp\bfg)_j$ denotes the $j$th component of $b^*\bfP^\perp\bfg$, $Z\sim N(0,1)$ and $\beta_0 \sim P_0$.
}
\end{rem}



\subsection{Asymptotic distributions}
In this section, we derive the asymptotic distribution of the MPLE. Let $\mathcal{S}_0$ be the set of the null components, i.e., $\mathcal{S}_0=\{1\leq j\leq p:\beta_j^*=0\}.$ 
\begin{theorem}\label{thm-test}
Suppose $\mathcal{S}\subseteq \mathcal{S}_0:=\{1\leq j\leq p:\beta_j^*=0\}$ and $|\mathcal{S}|=l$ is fixed in the asymptotics. Under Assumptions A1-A4, we have
\begin{align*}
\frac{\widehat{\bbeta}_{\mathcal{S}}}{b^*}\rightarrow^d N(0,\mathbf{I}_l),   
\end{align*}
where $\widehat{\bbeta}_{\mathcal{S}}=\{\widehat{\beta}_{j}:j\in\mathcal{S}\}$.
As a consequence,
\begin{align*}
\sum_{j\in \mathcal{S}}\left(\frac{\hat{\beta}_j}{b^*}\right)^2 \rightarrow^d \chi^2_l.    
\end{align*}
\end{theorem}
The above theorem shows that the MPLE of the null coefficients scaled by the constant $b^*$ converges to 
a multivariate normal distribution with the identity covariance matrix and hence the Wald test formed by the sum of squares of the MPLE converges to a chi-square distribution.

Below we conduct a simulation study to demonstrate the practical relevance of the finding in Theorem \ref{thm-test}.
Consider $n=500$, $p=200$ and half of the coordinates of $\vbeta$ are non-zero with $\kappa = 1$. Each non-zero component $\beta_j^*/c_{\kappa}'$ is independently generated from the uniform distribution on $[\kappa  - 1, \kappa+1]$, where the scaling parameter  $ c_{\kappa}'$ is set to $  \sqrt{ 2\kappa^2/(1/3+\kappa^2)}$ to keep the signal strength $\|\vbeta\|/\sqrt{p}$ equal to $\kappa$. We generate $50,000$ independent data sets and fit the Cox regression model to each data set. Figure~\ref{fig:pv-Gaussian-500-50000} (a) presents the two sided p-value $p_i=2\Phi(-|\hat{\beta}_j/b^*|)$ for the first 50 null coordinates of $\vbeta$ (combined over the $50,000$ simulation runs).
We also show the empirical cumulative distribution function (cdf) of $\Phi(\hat{\beta_j}/b^*)$ for a particular null coordinate of $\vbeta$ in Figure~\ref{fig:pv-Gaussian-500-50000} (b).
We observe that the p-values are uniformly distributed and there is a perfect agreement between the empirical cdf and the 45 degree line.

\begin{figure}[ht!]
    \centering
    \begin{subfigure}[t]{0.495\textwidth}
        \includegraphics[width=\textwidth]{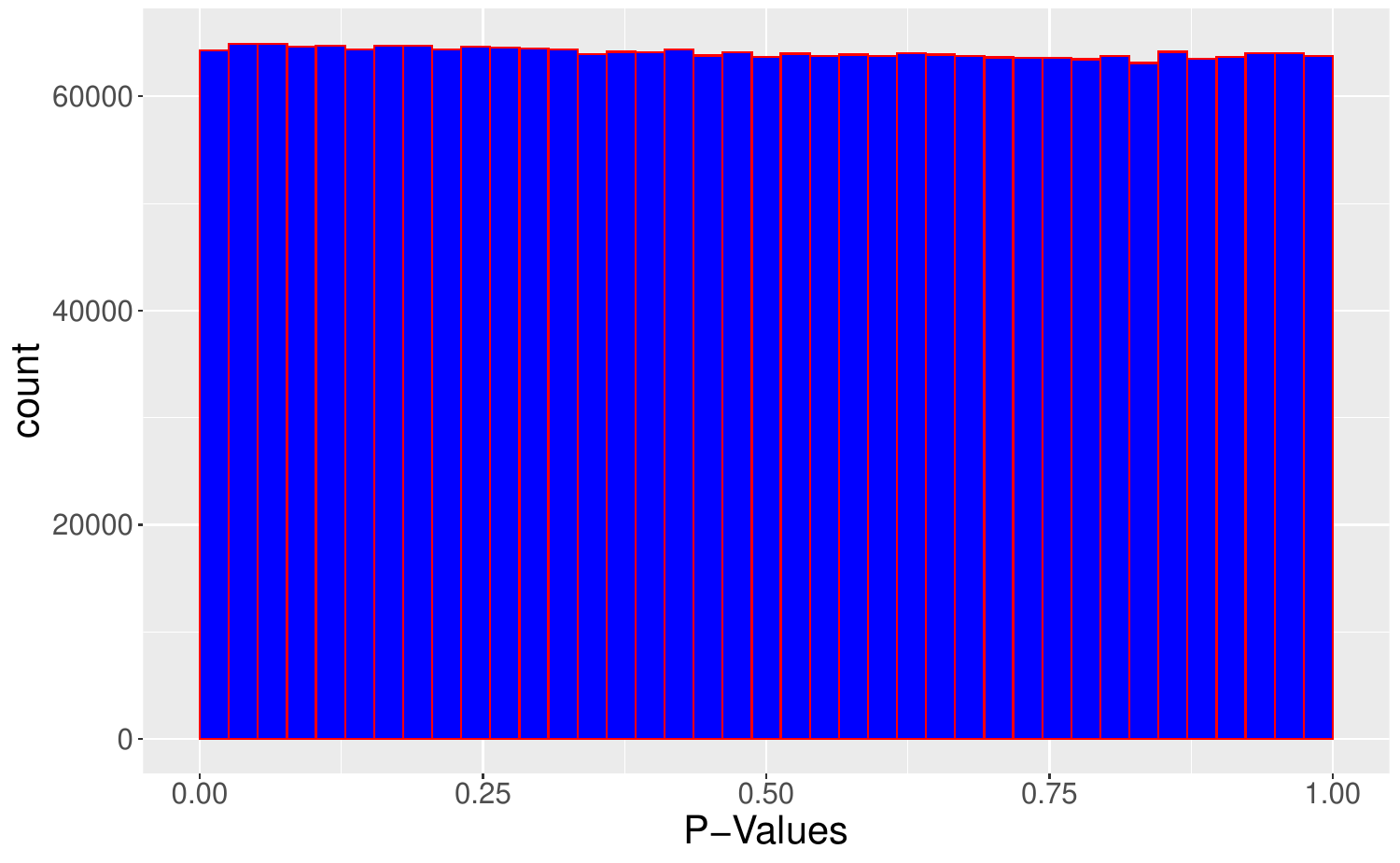}
        \caption{Histogram of the p-values for the first fifty null coordinates of $\vbeta$.}
    \end{subfigure}
    \hfill
    \begin{subfigure}[t]{0.495\textwidth}
        \includegraphics[width = \textwidth]{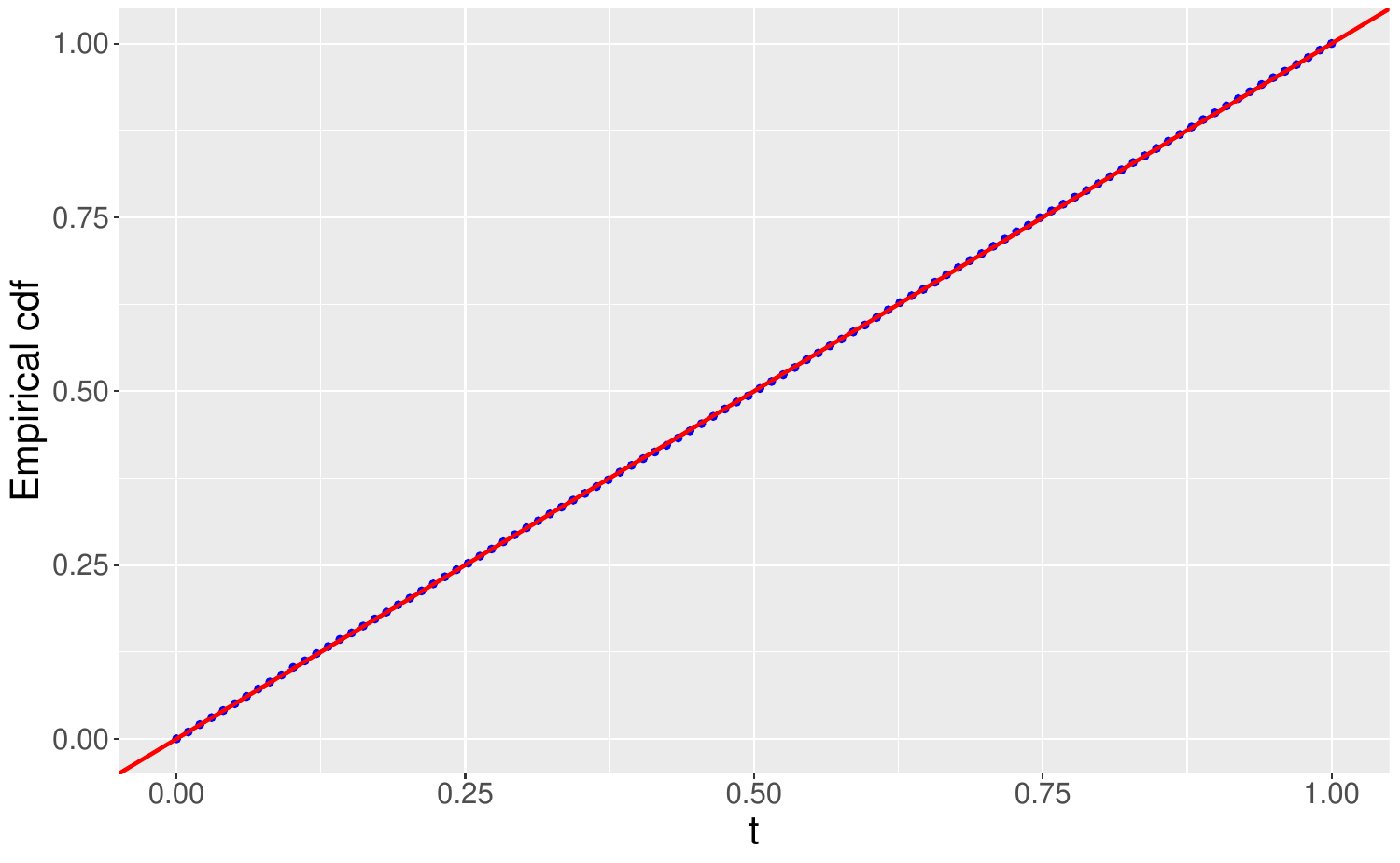}
        \caption{Empirical cdf of $\Phi(\hat{\beta}_j/b^*)$ for a particular null coordinate.}
    \end{subfigure}
    \caption{Histogram and empirical cdf, where $n=500$, $p=200$ and half of the coordinates of $\vbeta$ are non-zero with $\kappa = 1$. The results are based on $50,000$ replications.}
    \label{fig:pv-Gaussian-500-50000}
\end{figure}

Next we examine the chi-square approximation for the quantity $\sum_{j\in \mathcal{S}}(\hat{\beta}_j/b^*)^2$.
Figure \ref{fig:ChiCDF_50000} depicts the histograms for the p-value $p_i=F_l(\sum_{j\in \mathcal{S}}(\hat{\beta}_j/b^*)^2)$, where $\mathcal{S}\subseteq \mathcal{S}_0:=\{1\leq j\leq p:\beta_j^*=0\}$ with $|\mathcal{S}|=l=2,5$ and $F_l$ denotes the cdf for the chi-square distribution with $l$ degrees of freedom. The results suggest that the chi-square approximation is quite accurate. 


\begin{figure}[ht!]
    \centering
    \begin{subfigure}[t]{0.495\textwidth}
    \includegraphics[width = \textwidth]{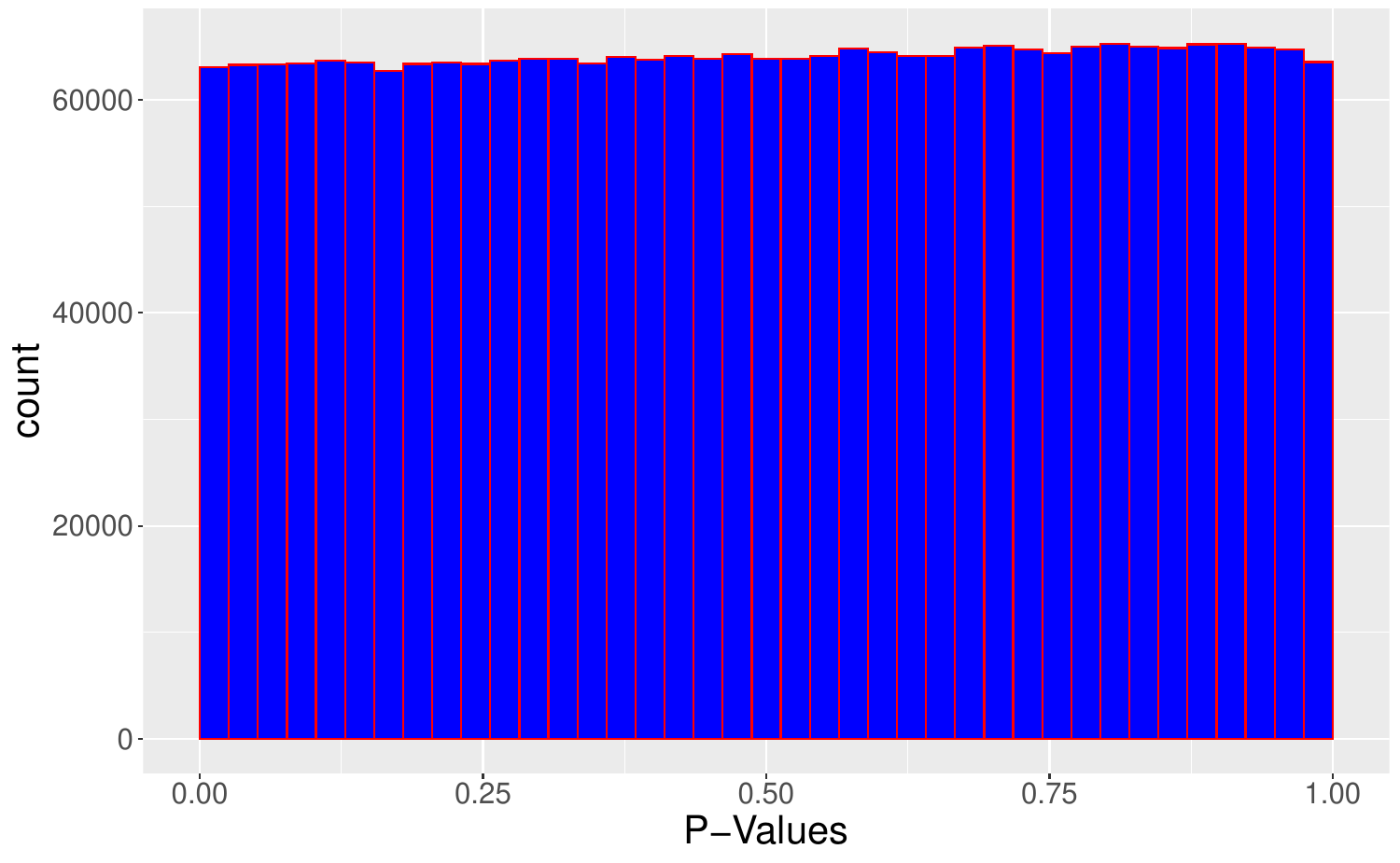}
    \caption{$l=2$}
    \end{subfigure}
    \hfill
    \begin{subfigure}[t]{0.495\textwidth}
    \includegraphics[width = \textwidth]{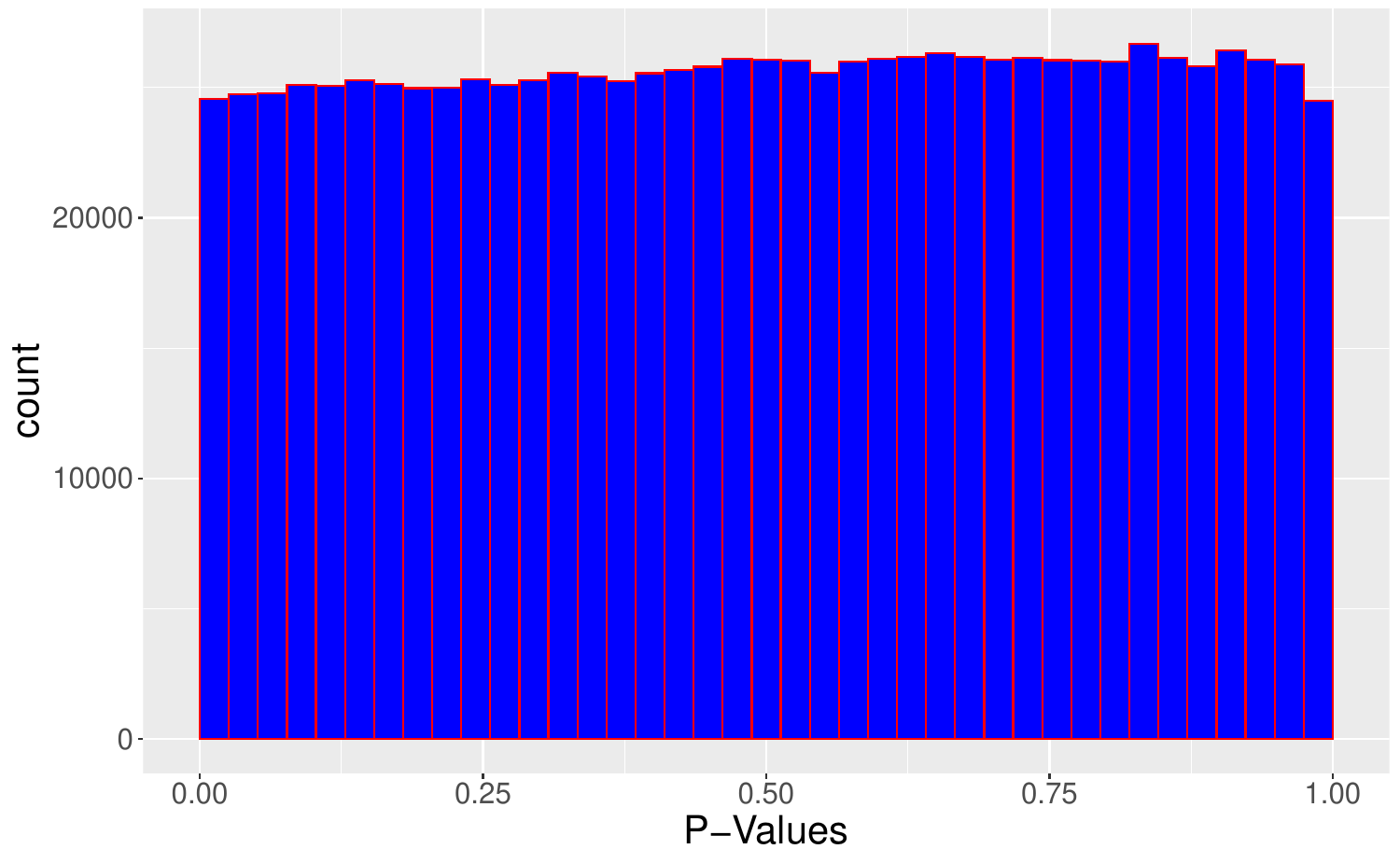}
    \caption{$l=5$}
    \end{subfigure}
    \caption{Histograms for $p_i=F_l(\sum_{j\in \mathcal{S}}(\hat{\beta}_j/b^*)^2)$ based on the chi-square approximation, where $n=500$, $p=200$ and half of the coordinates of $\vbeta$ are non-zero with $\kappa = 1$. The results are based on $50,000$ replications.}
    \label{fig:ChiCDF_50000}
\end{figure}

\section{Conclusion}\label{sec:future}
In this paper, we studied the asymptotic behavior of the MPLE in a high-dimensional Cox regression model with Gaussian covariates. We showed that the extence of the MPLE undergoes a sharp phase transition and we derived the explicit expression for phase transition boundary. In addition, we developed a new theory which gives the asymptotic distributions of the MPLE in the Cox regression model with independent Gaussian covariates. As a byproduct, we also obtained the limiting distribution for the Wald test. Our methods are built on some elements from convex geometry and the CGMT which is a modern version of the Gaussian comparison inequalities.

Finally we mention two future research directions. First, it would be interesting to investigate if the results derived in this paper hold for more general covariate distributions.
Second, it is of interest to study the penalized
regression problem
$\argmin_{\bbeta} L(\bbeta)+\rho(\bbeta),$
for some partial likelihood function $L(\cdot)$ and penalty function $\rho(\cdot)$. We leave these problems as future research topics.

\section*{Funding}
This work was supported by the National Natural Science Foundation of China (12201384 to HZ).

\newpage
\setcounter{section}{0}
\renewcommand{\thesection}{S\arabic{section}}
\setcounter{subsection}{5}
\renewcommand{\thesubsection}{S\arabic{subsection}}
\setcounter{equation}{0}
\renewcommand{\theequation}{S\arabic{equation}}
\setcounter{figure}{0}
\renewcommand{\thefigure}{S\arabic{figure}}
\setcounter{table}{0}
\renewcommand{\thetable}{S\arabic{table}}
\begin{center}
	\Huge{Supplementary Material}
\end{center}	

\section{Convex Gaussian Min-max Theorem}
\begin{defn}[GMT admissible sequence \citep{ThrampoulidisRegularized2015}]
	Let $\mathbf{G}\in\mathbb{R}^{n\times p}$, $\bfh\in\mathbb{R}^n$, $\bfg\in\mathbb{R}^p$, $\mathcal{S}_{\mathbf{w}}\subset\mathbb{R}^p$, $\mathcal{S}_{\mathbf{u}}\subset\mathbb{R}^n$, $\psi:\mathcal{S}_{\mathbf{w}}\times\mathcal{S}_{\mathbf{u}}\to\mathbb{R}$, all indexed by $p$ ($n=n(p)$). The sequence $\{\mathbf{G},\bfg,\bfh,\mathcal{S}_{\mathbf{w}},\mathcal{S}_{\mathbf{u}},\psi\}_{p\in\mathbb{N}}$, where $\mathbb{N}$ denotes the set of positive integers, is said to be admissible if for each $p\in\mathbb{N}$, $\mathcal{S}_{\mathbf{w}}$ and $\mathcal{S}_{\mathbf{u}}$ are compact sets and $\psi$ is continuous on its domain.
\end{defn}

A sequence $\{\mathbf{G},\bfg,\bfh,\mathcal{S}_{\mathbf{w}},\mathcal{S}_{\mathbf{u}},\psi\}_{p\in\mathbb{N}}$ defines a sequence of min-max problems:
\begin{align}
	\Phi(\mathbf{G}):=&\min_{\mathbf{w}\in\mathcal{S}_{\mathbf{w}}}\max_{\mathbf{u}\in\mathcal{S}_{\mathbf{u}}}\mathbf{u}^\top\mathbf{G}\mathbf{w}+\psi(\mathbf{w},\mathbf{u}), \label{eq-phi-G}\\
	\phi(\bfg,\bfh):=&\min_{\mathbf{w}\in\mathcal{S}_{\mathbf{w}}}\max_{\mathbf{u}\in\mathcal{S}_{\mathbf{u}}}\|\mathbf{u}\|_2\bfg^\top\mathbf{w}+\|\mathbf{w}\|_2\bfh^\top\mathbf{u}+\psi(\mathbf{w},\mathbf{u}). \label{eq-phi-g}
\end{align}
They are referred to as the Primary Optimization (PO) and Auxiliary Optimization (AO) problems, respectively. Denote the optimal minimizer
of (\ref{eq-phi-G}) as $\mathbf{w}_{\Phi}(\mathbf{G})$. 
Then the CGMT can be stated as follows.
\begin{theorem}[CGMT \citep{ThrampoulidisRegularized2015}]\label{thm1}
	Let $\{\mathbf{G},\bfg,\bfh,\mathcal{S}_{\mathbf{w}},\mathcal{S}_{\mathbf{u}},\psi\}_{p\in\mathbb{N}}$ be a GMT admissible sequence, for which additionally the entries of  $\mathbf{G}$, $\bfg$ and $\bfh$ are i.i.d. $N(0,1)$. The following four statements hold.\\
	(i) For any $p\in\mathbb{N}$ and $c\in\mathbb{R}$,
	\begin{align*}
		\mathbb{P}\{\Phi(\mathbf{G})<c\}\le 2\mathbb{P}\{\phi(\bfg,\bfh)<c\}.
	\end{align*}
(ii) Fix any $p\in\mathbb{N}$. If $\mathcal{S}_{\mathbf{w}}$, $\mathcal{S}_{\mathbf{u}}$ are convex sets, and $\psi(\cdot,\cdot)$ is convex-concave  (i.e., convex on its first argument and concave on its second argument) on $\mathcal{S}_{\mathbf{w}}\times\mathcal{S}_{\mathbf{u}}$, then, for any $\mu\in\mathbb{R}$ and $t>0$,
\begin{align*}
	\mathbb{P}\left\{|\Phi(\mathbf{G})-\mu|>t\right\}\le 2\mathbb{P}\left\{|\phi(\bfg,\bfh)-\mu|>t\right\}.
\end{align*}
(iii) Let $\mathcal{S}$ be an arbitrary open subset of $\mathcal{S}_{\mathbf{w}}$ and $\mathcal{S}^c=\mathcal{S}_{\mathbf{w}}\setminus \mathcal{S}$. Denote $\Phi_{\mathcal{S}^c}(\mathbf{G})$ and $\phi_{\mathcal{S}^c}(\bfg,\bfh)$ be the optimal
costs of the optimizations in (\ref{eq-phi-G}) and (\ref{eq-phi-g}), respectively, when the minimization over $\mathbf{w}$ is now
constrained over $\mathbf{w}\in\mathcal{S}^c$. If there exist constants $\bar{\phi}$, $\bar{\phi}_{\mathcal{S}^c}$, and $\eta>0$ such that,
\begin{enumerate}
    \item[a] $\bar{\phi}_{\mathcal{S}^c}\geq \bar{\phi}+3\eta$;
    \item[b] $\phi(\bfg,\bfh)<\bar{\phi}+\eta$ with probability at least $1-p$;
    \item[c] $\phi_{\mathcal{S}^c}(\bfg,\bfh)>\bar{\phi}_{\mathcal{S}^c}-\eta$ with probability at least $1-p$;
\end{enumerate}
Then we have $P(\mathbf{w}_{\Phi}(\mathbf{G})\in\mathcal{S})\geq 1-4p$. Here the probabilities are taken with respect to the randomness in $\mathbf{G},\bfg$ and $\bfh$.
\\(iv) Following the notation in (iii), suppose
there exist constants $\bar{\phi}<\bar{\phi}_{\mathcal{S}^c}$ such that 
$\phi(\bfg,\bfh)\rightarrow^p \bar{\phi}$ and $\phi_{\mathcal{S}^c}(\bfg,\bfh)\rightarrow^p \bar{\phi}_{\mathcal{S}^c}$. Then 
$$P(\mathbf{w}_{\Phi}(\mathbf{G})\in\mathcal{S})\rightarrow 1.$$

\end{theorem}

\section{A useful result from convex analysis}
Let $\mathcal{C}$ be a non-empty subset of $\mathbb{R}^n$. The polar cone of $\mathcal{C}$, denoted by $\mathcal{C}^*$, is defined as
\begin{align*}
\mathcal{C}^*=\{\mathbf{c}^*\in\mathbb{R}^n: \langle \mathbf{c}^*,\mathbf{c}\rangle\leq 0\text{ for all }\mathbf{c}\in\mathcal{C}\}.    
\end{align*}
We state the following result from the classical convex analysis. 
\begin{proposition}\citep{Moreau1962}\label{prop1}
Suppose $\mathcal{C}$ is a closed convex cone. For $\bfh\in\mathbb{R}^n$, let $\Pi_\mathcal{C}(\bfh)$ be the projection of $\bfh$ onto $\mathcal{C}$. Then we have the decomposition
\begin{align*}
    \bfh = \Pi_\mathcal{C}(\bfh) + \Pi_{\mathcal{C}^*}(\bfh),\quad \langle \Pi_\mathcal{C}(\bfh),\Pi_{\mathcal{C}^*}(\bfh)\rangle=0.    
    \end{align*}
As a consequence, $\langle \bfh- \Pi_{\mathcal{C}}(\bfh),\Pi_{\mathcal{C}}(\bfh)\rangle=0$ and hence $\|\bfh\|^2=\|\Pi_{\mathcal{C}}(\bfh)\|^2+\|\bfh-\Pi_{\mathcal{C}}(\bfh)\|^2.$
\end{proposition}

\begin{proposition}[Polar cone theorem]\label{prop2}
Suppose $\mathcal{C}$ is a closed convex cone. Then
\begin{align*}
    (\mathcal{C}^*)^*=\mathcal{C}.
\end{align*}
\end{proposition}

\begin{lemma}\label{lem1}
For two sets $\mathcal{C}_1$ and $\mathcal{C}_2$, we have $(\mathcal{C}_1+\mathcal{C}_2)^*=\mathcal{C}_1^*\cap \mathcal{C}_2^*.$
\end{lemma}


\section{Existence of the MLE in logistic regression: a revisit}\label{sec:MLE-logistic}
We revisit the logistic regression and reproduce the results in \citet{CandesPhase2018} using a new argument based on the CGMT. The new argument will be generalized to the Cox model in the next section. \citet{CandesPhase2018} considered the model
\begin{align*}
P(Y_i=1|\bfX_i)=1-P(Y_i=-1|\bfX_i)=\sigma(\beta_0^*+\bfX_i^\top \bbeta^*_1),\quad \sigma(x):=\frac{1}{1+\exp(-x)},\quad \bfX_i\sim N(0,\bSigma),
\end{align*}
where $\beta_0^*\in\mathbb{R}$ and $\bbeta^*_1\in\mathbb{R}^p$. Following their arguments, we have the equivalent model
$$(Y_i,\bfX_i)=^d(y_i,\bfq_i),$$
where 
\begin{align*}
& P(y_i=1|q_{i1})=1-P(y_i=-1|q_{i1})=\sigma(\beta_0+\gamma_0q_{i1}),\\
& (q_{i2},\dots,q_{ip})\sim N(0,\mathbf{I}_{p-1}),\\
& (q_{i2},\dots,q_{ip})\perp(y_i,q_{i1}).
\end{align*}
The MLE does not exist if and only if there exist $b_0\in\mathbb{R}$ and $\bb_1\in\mathbb{R}^p$ such that $(b_0,\bb_1^\top)^\top\neq 0$ and
\begin{align*}
y_i(b_0+\bfq^\top_i \bb_1)\geq 0,
\end{align*}
for all $1\leq i\leq n$. Due to the independence between $(y_i,q_{i1})$ and $(q_{i2},\dots,q_{ip})$, we have $y_i(q_{i2},\dots,q_{ip})\sim N(0,\mathbf{I}_{p-1})$. Let $\mathbf{V}$ be a $n\times 2$ matrix with the $i$th row being $(y_i,y_iq_{i1})$ and $\mathbf{Q}$ be a $n\times (p-1)$ matrix with the $i$th row being $y_i(q_{i2},\dots,q_{ip})\sim N(0,\mathbf{I}_{p-1})$, which is independent of $\mathbf{V}$. For a vector $\bfa=(a_1,\dots,a_p)^\top$, we write $\bfa\geq c$ (or $\bfa\leq c$) if $a_i\geq c$ (or $a_i\leq c$) for all $1\leq i\leq p$. 
To examine the existence of the MLE, we
fix any $\bfu>0$ and
consider the convex optimization problem
\begin{align*}
\max_{-1\leq \bb_1\leq 1,-1\leq \bb_2\leq 1} \bfu^\top (\bfV\bb_1 + \bfQ\bb_2) \quad \text{subject to}   \quad \bfV\bb_1 + \bfQ\bb_2 \geq 0,
\end{align*}
where $\bb_1\in\mathbb{R}^{2}$ and $\bb_2\in\mathbb{R}^{p-1}$. Notice that the MLE exists if and only if the optimal value of the objective function is equal to zero. 
We rewrite the problem in the Lagrangian form as
\begin{align}
\Phi(\mathbf{V},\mathbf{Q})=&\max_{-1\leq \bb_1\leq 1,-1\leq \bb_2\leq 1}\min_{\bfv\geq 0} \bfu^\top(\bfV\bb_1 + \bfQ\bb_2)+\bfv^\top(\bfV\bb_1 + \bfQ\bb_2 ) \label{eq-obj1}
\\=&\min_{\bfv\geq 0}\max_{-1\leq \bb_1\leq 1,-1\leq \bb_2\leq 1} (\bfu+\bfv)^\top\bfV\bb_1 +(\bfu+\bfv)^\top \bfQ\bb_2, \nonumber
\end{align}
where we can switch the order of the maximization and minimization as the objective function in (\ref{eq-obj1}) is concave-convex in its arguments. By the CGMT, we consider an asymptotically equivalent problem of the form
\begin{align*}
\phi(\mathbf{V},\mathbf{g},\mathbf{h})=\min_{\bfv\geq 0}\max_{-1\leq \bb_1\leq 1,-1\leq \bb_2\leq 1} &(\bfu+\bfv)^\top\bfV\bb_1 +\|\bfu+\bfv\| \bfg^\top \bb_2
-\|\bb_2\|\bfh^\top (\bfu+\bfv),
\end{align*}
where $\bfg\in\mathbb{R}^{p-1}$ and $\bfh\in\mathbb{R}^{n}$ both have i.i.d $N(0,1)$ components. 
Taking maximization with respect to the directions of $\bb_1$ and $\bb_2$, we obtain
\begin{align}\label{eq-AO-exist}
\min_{\bfv\geq 0} \max_{\|\bb_1\|,\|\bb_2\|}&\|\bfV^\top (\bfu+\bfv)\|\|\bb_1\| +\|\bfu+\bfv\| \|\bfg\| \|\bb_2\|
-\|\bb_2\|\bfh^\top (\bfu+\bfv). 
\end{align}
We observe two facts: 
\begin{itemize}
    \item[a.] $\Phi(\mathbf{V},\mathbf{Q})\geq 0$ and $\phi(\mathbf{V},\mathbf{g},\mathbf{h})\geq 0$;
    \item[b.] $P(\Phi(\mathbf{V},\mathbf{Q})>0)=P(\text{MLE does not exist})$.
\end{itemize}
Setting $\mu=0$ in (ii) of Theorem \ref{thm1} and using fact (a), we have 
\begin{align}\label{eq-ne}
 P(\Phi(\mathbf{V},\mathbf{Q})>t)\leq 2P(\phi(\mathbf{V},\bfg,\bfh)>t).   
\end{align}
Letting $t\downarrow 0$ and using fact (b), we get
\begin{align}\label{eq-e}
 P(\text{MLE does not exist})=P(\Phi(\mathbf{V},\mathbf{Q})>0)\leq 2P(\phi(\mathbf{V},\bfg,\bfh)>0). 
\end{align}
On the other hand, letting $c\downarrow 0$ in (i) of Theorem \ref{thm1}, we obtain 
\begin{align*}
P(\text{MLE exists})=P(\Phi(\mathbf{V},\mathbf{Q})=0)\leq 2P(\phi(\mathbf{V},\bfg,\bfh)=0).
\end{align*}

Next we introduce some notation. Define
\begin{align*}
\mathcal{C}=\text{span}(\bfV) + \{\bb:\bb\leq 0\}=\{\bfa+\bb:\bfa\in \text{span}(\bfV),\bb\leq 0\},
\end{align*}
where $\text{span}(\bfV)$ is the space spanned by the columns of $\bfV$. Let 
$$\mathcal{C}^*=\text{span}^\perp (\bfV)\cap \{\bb:\bb\geq 0\},$$
where $\text{span}^\perp (\bfV)$ denotes the orthogonal complement of $\text{span}(\bfV)$. Note that $\text{span}^\perp (\bfV)$ is also the polar cone of $\text{span}(\bfV)$ and $\{\bb:\bb\geq 0\}$ is the polar cone of $ \{\bb:\bb\leq 0\}$. Using Lemma \ref{lem1}, $\mathcal{C}^*$ 
is the polar cone of $\mathcal{C}$.
By Proposition \ref{prop1}, we have
\begin{align*}
\|\Pi_{\mathcal{C}^*}(\bfh)\|^2=\|\bfh-\Pi_{\mathcal{C}}(\bfh)\|^2=&\min_{\bfa\in \text{span}(\bfV),\bb\leq 0}\|\bfh-\bfa-\bb\|^2 =   \min_{\bfa\in \text{span}(\bfV)}\|(\bfh-\bfa)_+\|^2.
\end{align*}
As $p/n\rightarrow \delta$, by the laws of large numbers, we have
\begin{align*}
& \frac{1}{n}\|\Pi_{\mathcal{C}^*}(\bfh)\|^2=\frac{1}{n}\min_{\bfa\in \text{span}(\bfV)}\|(\bfh-\bfa)_+\|^2\rightarrow^p \min_{t_1,t_2\in\mathbb{R}}E[(h-t_1y-t_2yq)_+^2],\\
& \frac{1}{n}\|\bfg\|^2\rightarrow^p \delta,
\end{align*}
where $(y,q)$ has the same distribution as that of $(y_i,q_{i1})$. Below we consider two cases.

\textbf{Case 1:} Assuming $\delta>\min_{t_1,t_2\in\mathbb{R}}E[(h-t_1y-t_2yq)_+^2]$, we aim to show that $P(\text{MLE does not exist})\rightarrow 1$. In this case, we have 
\begin{align*}
\frac{1}{\sqrt{n}}\|\Pi_{\mathcal{C}^*}(\bfh)\|< \frac{1}{\sqrt{n}}\|\bfg\|
\end{align*}
with probability tending to one. For the objective function in (\ref{eq-AO-exist}) to be zero, we need to find 
a vector $\bfv\geq 0$ such that $$\bfV^\top (\bfu+\bfv)=0\quad 
\text{and}\quad 
\frac{\bfh^\top(\bfu+\bfv)}{\|\bfu+\bfv\| }\ge \|\bfg\|.$$
However, this event happens with probability tending to zero as when $\bfV^\top (\bfu+\bfv)=0$, we have $\bfu+\bfv\in\mathcal{
C}^*$ and 
\begin{align*}
P\left(\frac{1}{\sqrt{n}}\|\bfg\|>\frac{1}{\sqrt{n}}\|\Pi_{\mathcal{C}^*}(\bfh)\| \geq \frac{1}{\sqrt{n}}\frac{\bfh^\top(\bfu+\bfv)}{\|\bfu+\bfv\| }\right)\rightarrow 1.     
\end{align*}
Therefore, we must have 
$$P(\text{MLE exists})=P(\Phi(\mathbf{V},\mathbf{Q})=0)\leq 2P(\phi(\mathbf{V},\bfg,\bfh)=0)\rightarrow 0,$$
which implies that $P(\text{MLE does not exist})\rightarrow 1.$

\textbf{Case 2:} Assuming $\delta<\min_{t_1,t_2\in\mathbb{R}}E[(h-t_1y-t_2yq)_+^2]$, we show that $P(\text{MLE exists})\rightarrow 1$.
Let $\mathbb{R}^n_+=\{\bb\in\mathbb{R}^n:\bb\geq 0\}$ and $\widetilde{\mathbb{R}}_+^n$ be the interior of $\mathbb{R}^n_+.$ 
By similar arguments as in Lemma 2 of \citet{CandesPhase2018}, we have 
\begin{align}\label{eq-proj}
P(\text{span}^\perp(\mathbf{V}) \cap \widetilde{\mathbb{R}}_+^n\neq \emptyset)\rightarrow 1.
\end{align}
We note that for any $\bfu\in \widetilde{\mathbb{R}}_+^n$, 
\begin{align}\label{eq-direct}
\left\{\frac{\bfu+\bfv}{\|\bfu+\bfv\|}:\bfv\geq 0\right\}=\left\{\frac{\bfv}{\|\bfv\|}:\bfv\in \widetilde{\mathbb{R}}_+^n \right\}.
\end{align}
With probability converging to one, the projection of $\bfh$ onto $\mathcal{C}^*$ is in $\widetilde{\mathbb{R}}_+^n$. Using (\ref{eq-proj}), with high probability, we can find $\widetilde{\bfv}\in  \text{span}^\perp(\mathbf{V})\cap \widetilde{\mathbb{R}}_+^n$ such that  $\|\Pi_{\mathcal{C}^*}(\bfh)\|=\bfh^\top \widetilde{\bfv}/\|\widetilde{\bfv}\|$. By (\ref{eq-proj}) and (\ref{eq-direct}), there exists a $\bfv^*\geq 0$ satisfying that $\bfu+\bfv^*\in \text{span}^\perp(\mathbf{V})\cap \widetilde{\mathbb{R}}_+^n$ and
$$\|\Pi_{\mathcal{C}^*}(\bfh)\|=\frac{\bfh^\top(\bfu+\bfv^*)}{\|\bfu+\bfv^*\| }.$$
Under the assumption that $\delta<\min_{t_1,t_2\in\mathbb{R}}E[(h-t_1y-t_2yq)_+^2]$,
$$\frac{1}{\sqrt{n}}\|\Pi_{\mathcal{C}^*}(\bfh)\|>\frac{1}{\sqrt{n}}\|\bfg\|$$ with probability tending to one, which implies that $\phi(\mathbf{V},\mathbf{g},\mathbf{h})=0$ when $\bfv$ in (\ref{eq-AO-exist}) is chosen to be $\bfv^*$. Therefore, we obtain
$$ P(\text{MLE does not exist})=P(\Phi(\mathbf{V},\mathbf{Q})>0)\leq 2P(\phi(\mathbf{V},\bfg,\bfh)>0)\rightarrow 0,$$
or equivalently $P(\text{MLE exists})\rightarrow 1$.

\section{Existence of the MPLE in Cox regression}
Under the Cox model (\ref{eq-lambda}), the conditional distribution of $Y_i$ given $\bfX_i$ depends on $\bfX_i$ only through a linear combination $\bfX_i^\top \bbeta^*$.
By the rotational invariance of the Gaussian distribution, we can show that the joint distribution of $(Y_i,\bfX^\top_i)=(Y_i,X_{i1},\dots,X_{ip})$ is the same as that of $$(y_i,\bfq_i^\top)=(y_i,q_{i1},\dots,q_{ip}),$$ 
where $y_i=t_i\wedge C_i$ with $t_i$ having the hazard function 
$$\lambda(t|q_{i1})=\lambda_0(t)\exp(\kappa q_{i1})$$ 
and 
\begin{align*}
&(q_{i2},\dots,q_{ip})\sim N(0,\mathbf{I}_{p-1}),\\ &(q_{i2},\dots,q_{ip})\perp(y_i,q_{i1}),
\end{align*}
for $1\leq i\leq n.$ Here $(y_i,\kappa q_{i1})$ has the same distribution as that of $(Y_i,\bfX_i^\top\bbeta^*)$.
Thus we only need to study the existence of the MPLE in the equivalent model.
To this end, we consider the convex optimization problem
\begin{align*}
&\max_{-1\leq \bb\leq 1}\sum_{(i,j)\in\mathcal{D}}a_{ij}\bb^\top(\bfq_i-\bfq_j)\\ 
&\text{subject to} \quad \bb^\top(\bfq_i-\bfq_j)\geq 0 \text{ for all }(i,j)\in\mathcal{D},
\end{align*}
where $a_{ij}> 0$ is fixed throughout the arguments and $\bb=(b_1,\dots,b_p)^\top$. The MPLE exists if and only if the optimal value of the above objective function is equal to zero.  
Define $\widetilde{\bfa}=(\widetilde{a}_1,\dots,\widetilde{a}_n)^\top$ with $\widetilde{a}_i=\sum_{j=1}^n a_{ij}\Delta_i\mathbf{1}\{Y_j\geq Y_i\}$ and
$\breve{\bfa}=(\breve{a}_1,\dots,\breve{a}_n)^\top$ with $\breve{a}_j=\sum_{i=1}^n a_{ij}\Delta_i\mathbf{1}\{Y_j\geq Y_i\}$. Let $\bfQ=(\bfq_1,\dots,\bfq_n)^\top=(\widetilde{\bfq}_1,\bfQ_2)\in\mathbb{R}^{n\times p},$ where $\widetilde{\bfq}_1=(q_{11},\dots,q_{n1})^\top\in\mathbb{R}^n$ and $\bfQ_2\in\mathbb{R}^{n\times (p-1)}$.
Note that
\begin{align}\label{eq-obj2}
\sum_{(i,j)\in\mathcal{D}}a_{ij}\bb^\top(\bfq_i-\bfq_j)
=&\sum_{i=1}^n\sum_{j=1}^n a_{ij}\Delta_i\mathbf{1}\{Y_j\geq  Y_i\}\bb^\top(\bfq_i-\bfq_j)
=(\widetilde{\bfa}-\breve{\bfa})^\top \widetilde{\bfq}_1b_1+(\widetilde{\bfa}-\breve{\bfa})^\top\bfQ_2\bb_2,
\end{align}
for $\bb_2=(b_2,\dots,b_p)^\top.$ By introducing the Lagrangian $\{v_{ij}\}_{(i,j)\in\mathcal{D}}$ and using (\ref{eq-obj2}), we can rewrite the optimization problem as 
\begin{align}
\Phi(\widetilde{\bfq}_1,\mathbf{Q}_2)=&\max_{-1\leq \bb\leq 1}\min_{v_{ij}\geq0}\sum_{(i,j)\in\mathcal{D}}a_{ij}\bb^\top(\bfq_i-\bfq_j)+\sum_{(i,j)\in\mathcal{D}}v_{ij}\bb^\top(\bfq_i-\bfq_j) \nonumber \\
=&\max_{-1\leq b_1\leq 1,-1\leq \bb_2\leq 1}\min_{v_{ij}\geq 0}(\widetilde{\bfa}-\breve{\bfa}+\widetilde{\bfv}-\breve{\bfv})^\top \widetilde{\bfq}_1b_1+(\widetilde{\bfa}-\breve{\bfa}+\widetilde{\bfv}-\breve{\bfv})^\top\bfQ_2\bb_2 \label{eq-obj3}\\
=&\min_{v_{ij}\geq 0}\max_{-1\leq b_1\leq 1,-1\leq \bb_2\leq 1}(\widetilde{\bfa}-\breve{\bfa}+\widetilde{\bfv}-\breve{\bfv})^\top \widetilde{\bfq}_1b_1+(\widetilde{\bfa}-\breve{\bfa}+\widetilde{\bfv}-\breve{\bfv})^\top\bfQ_2\bb_2,\nonumber
\end{align}
where $\widetilde{\bfv}$ and $\breve{\bfv}$ are defined in a similar way as $\widetilde{\bfa}$ and $\breve{\bfa}$ with $a_{ij}$ replaced by $v_{ij}$. Here we switch the order of the maximization and minimization as the objective function in (\ref{eq-obj3}) is concave-convex. Conditional on $\{y_i,q_{i1},\Delta_i\}_{i=1}^n$ and using the CGMT, we consider an asymptotically equivalent AO problem defined as 
\begin{equation}\label{eq-AO-exist2}
\begin{split}
\phi(\widetilde{\mathbf{q}}_1,\mathbf{g},\mathbf{h})=&\min_{v_{ij}\geq 0}\max_{-1\leq b_1\leq 1,-1\leq \bb_2\leq 1}(\widetilde{\bfa}-\breve{\bfa}+\widetilde{\bfv}-\breve{\bfv})^\top \widetilde{\bfq}_1b_1+\|\widetilde{\bfa}-\breve{\bfa}+\widetilde{\bfv}-\breve{\bfv}\| \bfg^\top \bb_2
\\&-\|\bb_2\|\bfh^\top (\widetilde{\bfa}-\breve{\bfa}+\widetilde{\bfv}-\breve{\bfv}),
\end{split}
\end{equation}
where $\bfg\in\mathbb{R}^{p-1}$ and $\bfh\in\mathbb{R}^{n}$ both have i.i.d $N(0,1)$ components that are independent of other random quantities. Taking maximization with respect to the directions of $b_1$ and $\bb_2$, the optimization problem becomes
\begin{align*}
\min_{v_{ij}\geq 0}\max_{\|\bb_2\|\leq \sqrt{p-1}}|(\widetilde{\bfa}-\breve{\bfa}+\widetilde{\bfv}-\breve{\bfv})^\top \widetilde{\bfq}_1|+\|\widetilde{\bfa}-\breve{\bfa}+\widetilde{\bfv}-\breve{\bfv}\|\|\bfg\| \|\bb_2\|-\|\bb_2\|\bfh^\top (\widetilde{\bfa}-\breve{\bfa}+\widetilde{\bfv}-\breve{\bfv}).
\end{align*}

Recall that 
\begin{align*}
\mathcal{M}=\left\{\mathbf{m}=(m_1,\dots,m_n)\in\mathbb{R}^n: \min_{s<i_l}m_{s}\geq m_{i_l}, \max_{j\in D_{i_l}}m_j\leq m_{i_l}, l=1,\dots,k\right\},
\end{align*}
where $D_{i_l}=\{1\leq j<i_l: y_j=y_{i_l},\Delta_j=1\}$. Define the set
\begin{align*}
\mathcal{M}^*=\left\{\mathbf{m}^*=(m_1^*,\dots,m_n^*)\in\mathbb{R}^n: m_i^*=\sum_{j=1}^n \left(c_{ij}\Delta_i\mathbf{1}\{y_j\geq y_i\}- c_{ji}\Delta_j\mathbf{1}\{y_i\geq  y_j\}\right)\text{ for some }c_{ij},c_{ji}\geq 0
\right\},  
\end{align*}
which will be shown to be the polar cone of $\mathcal{M}$. Further let
\begin{align*}
&\mathcal{C}=\text{span}(\widetilde{\bfq}_1) + \mathcal{M}=\{t\widetilde{\bfq}_1+\mathbf{m}: t\in\mathbb{R},\mathbf{m}\in\mathcal{M}\}\quad \text{and}\quad  \mathcal{C}^*=\text{span}^\perp(\widetilde{\bfq}_1)\cap \mathcal{M}^*.
\end{align*}
It is not hard to see that $\mathcal{M}^*$ is a cone. Next we show that $\mathcal{M}$ is the polar cone of $\mathcal{M}^*$.
For any $\mathbf{m}=(m_1,...,m_n)^\top$ in the polar cone of $\mathcal{M}^*$ and $\mathbf{m}^*\in\mathcal{M}^*$, we must have
\begin{align*}
\langle\mathbf{m},\mathbf{m}^* \rangle=&\sum_{1\leq i\neq j\leq n}m_i\left(c_{ij}\Delta_i\mathbf{1}\{y_j\geq  y_i\}- c_{ji}\Delta_j\mathbf{1}\{y_i\geq  y_j\}\right)
\\=&\sum_{1\leq i<j\leq n}(m_i-m_j)\left(c_{ij}\Delta_i\mathbf{1}\{y_j\geq  y_i\}- c_{ji}\Delta_j\mathbf{1}\{y_i\geq  y_j\}\right)\leq 0,
\end{align*}
for any $c_{ij}\ge 0$. Set $c_{kl}=c_{lk}=0$ if $\{k,l\}\neq\{i,j\}$. We have
\begin{align*}
(m_i-m_j)\left(c_{ij}\Delta_i\mathbf{1}\{y_j\geq  y_i\}- c_{ji}\Delta_j\mathbf{1}\{y_i\geq  y_j\}\right)\leq 0.
\end{align*}
We see that
\begin{align*}
\begin{cases}
m_i\leq m_j \quad &\Delta_i\mathbf{1}\{y_j\geq  y_i\}=1,\Delta_j\mathbf{1}\{y_i\geq y_j\}=0,\\
m_i\geq m_j \quad &\Delta_i\mathbf{1}\{y_j\geq  y_i\}=0,\Delta_j\mathbf{1}\{y_i\geq y_j\}=1,\\
m_i=m_j \quad &\Delta_i\mathbf{1}\{y_j\geq  y_i\}=1,\Delta_j\mathbf{1}\{y_i\geq y_j\}=1,\\
\text{no restriction} \quad &\Delta_i\mathbf{1}\{y_j\geq  y_i\}=0,\Delta_j\mathbf{1}\{y_i\geq y_j\}=0,
\end{cases}    
\end{align*}
which implies that $\mathbf{m}\in\mathcal{M}$. On the other hand, it is not hard to see that for any $\mathbf{m}\in\mathcal{M}$ and $\mathbf{m}^*\in\mathcal{M}^*$, $\langle\mathbf{m},\mathbf{m}^* \rangle\leq0$. Thus $\mathcal{M}$ is the polar cone of $\mathcal{M}^*$. 
As $\mathcal{M}^*$ is closed and convex,
by Proposition \ref{prop2}, we obtain that $\mathcal{M}^*$ is the polar cone of $\mathcal{M}$. As $\text{span}^\perp(\widetilde{\bfq}_1)$ is the polar cone of $\text{span}(\widetilde{\bfq}_1)$, using Lemma \ref{lem1}, we have that $\mathcal{C}^*$ is the polar cone of $\mathcal{C}$. Similar to the discussions for logistic regression, we consider two cases.

\textbf{Case 1:} Suppose $\delta>h_U(\lambda_0,\kappa,P_\mathcal{C})$. We show that $P(\text{MPLE does not exist})\rightarrow 1$. By the assumption, we have
\begin{align}\label{eq-ineq-1}
\frac{1}{\sqrt{n}}\|\bfh-\Pi_{\mathcal{C}}(\bfh)\|=\frac{1}{\sqrt{n}}\|\Pi_{\mathcal{C}^*}(\bfh)\|< \frac{1}{\sqrt{n}}\|\bfg\|
\end{align}
with probability tending to one. For the objective function in (\ref{eq-AO-exist2}) to be zero, we need to find $\{v_{ij}^*\}_{(i,j)\in\mathcal{D}}$ such that
\begin{align}\label{eq-con}
&(\widetilde{\bfa}-\breve{\bfa}+\widetilde{\bfv}^*-\breve{\bfv}^*)^\top \widetilde{\bfq}_1=0\quad \text{and}\quad \frac{\bfh^\top (\widetilde{\bfa}-\breve{\bfa}+\widetilde{\bfv}^*-\breve{\bfv}^*)}{\|\widetilde{\bfa}-\breve{\bfa}+\widetilde{\bfv}^*-\breve{\bfv}^*\|}\ge\|\bfg\|.
\end{align}
The definitions of $\widetilde{\bfa},\breve{\bfa},\widetilde{\bfv}^*$ and $\breve{\bfv}^*$ imply that $\widetilde{\bfa}-\breve{\bfa}+\widetilde{\bfv}^*-\breve{\bfv}^*\in \mathcal{M}^*$. Moreover, 
from the first condition in (\ref{eq-con}), we have $\widetilde{\bfa}-\breve{\bfa}+\widetilde{\bfv}^*-\breve{\bfv}^*\in \text{span}^\perp (\widetilde{\bfq}_1)$ and thus $\widetilde{\bfa}-\breve{\bfa}+\widetilde{\bfv}-\breve{\bfv}\in \mathcal{C}^*.$ By (\ref{eq-ineq-1}), we get
\begin{align*}
P\left(\frac{1}{\sqrt{n}}\|\bfg\|>\frac{1}{\sqrt{n}}\|\Pi_{\mathcal{C}^*}(\bfh)\| \geq \frac{1}{\sqrt{n}}\frac{\bfh^\top (\widetilde{\bfa}-\breve{\bfa}+\widetilde{\bfv}^*-\breve{\bfv}^*)}{\|\widetilde{\bfa}-\breve{\bfa}+\widetilde{\bfv}^*-\breve{\bfv}^*\|}\right)\rightarrow 1.    
\end{align*}
Therefore, we must have 
$$P(\text{MLE exists})=P(\Phi(\widetilde{\bfq}_1,\mathbf{Q})=0)\leq 2P(\phi(\widetilde{\mathbf{q}}_1,\bfg,\bfh)=0)\rightarrow 0,$$
which implies that $P(\text{MLE does not exist})\rightarrow 1.$

\textbf{Case 2:} Assuming that $\delta<h_L(\lambda_0,\kappa,P_\mathcal{C})$, we argue that $P(\text{MPLE exists})\rightarrow 1$. Let $\widetilde{\mathcal{M}}^*$ denote the interior of $\mathcal{M}^*$. 
We first claim that
\begin{align}\label{eq-obj4}
P(\text{span}^\perp (\widetilde{\bfq}_1)\cap \widetilde{\mathcal{M}}^*\neq\{{\bf 0}\})\rightarrow 1.
\end{align}
To see this, we note that for any $\mathbf{m}^*\in \widetilde{\mathcal{M}}^*$, 
\begin{align*}
\langle\mathbf{m}^*,\widetilde{\bfq}_1 \rangle=&\sum_{i=1}^n\sum_{j=1}^n q_i\left(c_{ij}\Delta_i\mathbf{1}\{y_j\geq y_i\}- c_{ji}\Delta_j\mathbf{1}\{y_i\geq  y_j\}\right)\\
=&\sum_{j<i}(q_i-q_j)\left(c_{ij}\Delta_i\mathbf{1}\{y_j\geq y_i\}- c_{ji}\Delta_j\mathbf{1}\{y_i\geq  y_j\}\right)\\
=&\sum_{j<i,y_i\neq y_j}(q_i-q_j)c_{ij}\Delta_i+\sum_{j<i,y_i=y_j}(q_i-q_j)\left(c_{ij}\Delta_i- c_{ji}\Delta_j\right).
\end{align*}
We note that the event that all $q_i-q_j$ with $j<i$, $y_i\neq y_j$ and $\Delta_i=1$ have the same sign happens with exponentially small probability. As $c_{ij}$'s are arbitrarily positive, the first term in the last equality can be equal to any value including the negative of the second term with appropriate choice of $c_{ij}$'s. Hence, with high probability, there exists $\mathbf{m}^*\in\widetilde{\mathcal{M}}^*$ such that $\langle\mathbf{m}^*,\widetilde{\bfq}_1 \rangle=0$, which justifies claim (\ref{eq-obj4}).
Also, it is not hard to verify that for any $\bfa\in \widetilde{\mathbb{R}}^n_+$, 
\begin{align}\label{eq-direct1}
\left\{\frac{\widetilde{\bfa}-\breve{\bfa}+\widetilde{\bfv}-\breve{\bfv}}{\|\widetilde{\bfa}-\breve{\bfa}+\widetilde{\bfv}-\breve{\bfv}\|}:\bfv\geq 0\right\}=\left\{\frac{ \widetilde{\bfv}-\breve{\bfv}}{\|\widetilde{\bfv}-\breve{\bfv}\|}:\bfv\in \widetilde{\mathbb{R}}^n_+ \right\}.
\end{align}
With high probability, the projection of $\bfh$ onto $\mathcal{C}^*$ is in $\widetilde{\mathcal{M}}^*$. Hence by (\ref{eq-obj4}) and (\ref{eq-direct1}), we can find $\bfv^*\geq 0$ such that $\widetilde{\bfa}-\breve{\bfa}+\widetilde{\bfv}^*-\breve{\bfv}^* \in \text{span}^\perp (\widetilde{\bfq}_1)$, $\|\widetilde{\bfa}-\breve{\bfa}+\widetilde{\bfv}^*-\breve{\bfv}^*\|\neq 0$, and
$$\|\Pi_{\mathcal{C}^*}(\bfh)\|=\frac{\bfh^\top (\widetilde{\bfa}-\breve{\bfa}+\widetilde{\bfv}^*-\breve{\bfv}^*)}{\|\widetilde{\bfa}-\breve{\bfa}+\widetilde{\bfv}^*-\breve{\bfv}^*\|}.$$
Under the assumption that $\delta<h_L(\lambda_0,\kappa,P_\mathcal{C})$,
$$\frac{1}{\sqrt{n}}\|\Pi_{\mathcal{C}^*}(\bfh)\|>\frac{1}{\sqrt{n}}\|\bfg\|$$ with probability tending to one. With the $\bfv^*$ chosen above, $\phi(\widetilde{\mathbf{q}}_1,\mathbf{g},\mathbf{h})=0.$ Therefore, we obtain
$$ P(\text{MLE does not exist})=P(\Phi(\widetilde{\bfq}_1,\mathbf{Q})>0)\leq 2P(\phi(\widetilde{\mathbf{q}}_1,\bfg,\bfh)>0)\rightarrow 0,$$
which implies that $P(\text{MLE exists})\rightarrow 1$.

\begin{rem}\label{rm-c}
{\rm
Suppose the censoring time $C_i$ depends on the covariate $\bfX_i\sim N(0,\mathbf{I}_p)$ through $\bfX_i^\top\boldsymbol{\theta}$ for some $\boldsymbol{\theta}\in\mathbb{R}^p$, and $C_i$ is  conditionally independent of the survival time $T_i$ given $\bfX_i$. Let $\bfA$ be an orthogonal matrix  with first row being $(\bbeta^*/\|\bbeta^*\|)^\top$ and second row being $\boldsymbol{\theta}^\top\bfP^\perp/\|\bfP^\perp\boldsymbol{\theta}\|$, where $\bfP=\bbeta^*\bbeta^{*\top}/\|\bbeta^*\|^2$. Let $\bfA\bfX_i=\bfq_i=(q_{i1},\dots,q_{ip})^\top$. We have $\bfX_i^\top\bbeta^*=\|\bbeta^*\|q_{i1}$ and 
$$\bfX_i^\top\boldsymbol{\theta}=\bfX_i^\top\bfP^\perp\boldsymbol{\theta}+\bfX_i^\top\bfP\boldsymbol{\theta}=\|\bfP^\perp\boldsymbol{\theta}\|q_{i2}+\bfX_i^\top\bbeta^*\boldsymbol{\theta}^\top\bbeta^*/\|\bbeta^*\|^2=\|\bfP^\perp\boldsymbol{\theta}\|q_{i2}+\|\bfP\boldsymbol{\theta}\|q_{i1}.$$
Thus we have found the equivalent model:
$$P_{Y_i,\bfX_i}(y,\mathbf{x})=P_{Y_i|\bfX_i}(y|\bfX_i^\top\bbeta^*,\bfX_i^\top\boldsymbol{\theta})P_{\bfX_i}(\mathbf{x})$$
which can be expressed equivalently as
$$P_{y_i,\bfq_i}(y,\mathbf{q})=P_{y_i|\bfq_i}(y|q_{i1},\|\bfP^\perp\boldsymbol{\theta}\|q_{i2}+\|\bfP\boldsymbol{\theta}\|q_{i1})P_{\bfq_i}(\mathbf{q})$$
through the rotation $\mathbf{A}$,
where 
\begin{align*}
&\bfq_i=(q_{i1},\dots,q_{ip})^\top\sim N(0,\mathbf{I}_p),\quad (q_{i3},\dots,q_{ip})\perp(y_i,q_{i1}, q_{i2}).
\end{align*}
Analogy to the case where $C_i$ is independent of $\bfX_i$, we can derive similar results by replacing $\widetilde{\bfq}_1,\bfQ_2$ with $\bfQ_1\in\mathbb{R}^{n\times 2}$ and $\bfQ_2\in\mathbb{R}^{n\times(p-2)}$.
The argument is alike if $C_i$ depends on multiple linear combinations of $\bfX_i$.
}
\end{rem}

\section{Error analysis of the MPLE}\label{sec:error-analysis}
\subsection*{Reformulating the PO}
Let $\bfH=\sqrt{p}(\bfX_1,\dots,\bfX_n)^\top$ and $\kappa=\|\bbeta^*\|/\sqrt{p}$. Recall the definition of the partial log-likelihood $L$ in (\ref{PL}). We can 
express it as
\begin{align*}
L(\bbeta)= \frac{1}{n}\bDelta^\top\bfu - \frac{1}{n}\bDelta^\top \log\left(\bfA \exp(\bfu)\right),  
\end{align*}
where $$\bfu=\frac{1}{\sqrt{p}}\bfH\bbeta,$$  $\bDelta=(\Delta_1,\dots,\Delta_n)^\top$ and $\bfA=n^{-1}(\bfa_1,\dots,\bfa_n)^\top$ with $\bfa_i=(\mathbf{1}\{Y_1\geq Y_i\},\dots,\mathbf{1}\{Y_n\geq Y_i\})^\top$.
By introducing a Lagrange multiplier $\bfv$, we can write the optimization problem in (\ref{PL}) as a min-max optimization
\begin{align}\label{AO-1}
&\min_{\bbeta\in\mathbb{R}^p,\bfu\in\mathbb{R}^n}\max_{\bfv\in\mathbb{R}^n}-\frac{1}{n}\bDelta^\top \bfu + \frac{1}{n} \bDelta^\top \log\left(\bfA \exp(\bfu)\right)+\frac{\bfv^\top}{n}\left(\bfu-\frac{1}{\sqrt{p}}\bfH\bbeta\right).
\end{align}
The bilinear form $\bfv^\top \bfH \bbeta$ depends on $\bDelta$ and $\mathbf{A}$. Define $\bfP=\bbeta^*\bbeta^{*\top}/\|\bbeta^*\|^2$ and 
$\bfP^\perp=I-\bfP.$ We have
$$\bfH=\bfH_1+\bfH_2,\quad \bfH_1=\bfH \bfP,\quad \bfH_2=\bfH\bfP^\perp.$$
With the above decomposition, (\ref{AO-1}) can be rewritten as 
\begin{align*}
&\min_{\bbeta\in\mathbb{R}^p,\bfu\in\mathbb{R}^n}\max_{\bfv\in\mathbb{R}^n}-\frac{1}{n}\bDelta^\top\bfu + \frac{1}{n} \bDelta^\top \log\left(\bfA \exp(\bfu)\right)+\frac{1}{n}\bfv^\top\left(\bfu-\frac{1}{\sqrt{p}}\bfH_1\bbeta\right)-\frac{1}{n\sqrt{p}}\bfv^\top\bfH_2\bfP^\perp\bbeta.   
\end{align*}
We note that the objective function above is convex with respect
to $\bbeta$ and $\bfu$ and concave with respect to $\bfv$. Using the CGMT, we consider the AO problem defined as
\begin{equation}\label{AO-2}
\begin{split}
\min_{\bbeta\in\mathbb{R}^p,\bfu\in\mathbb{R}^n}\max_{\bfv\in\mathbb{R}^n}&-\frac{1}{n} \bDelta^\top\bfu + \frac{1}{n} \bDelta^\top \log\left(\bfA \exp(\bfu)\right)+\frac{1}{n}\bfv^\top\left(\bfu-\frac{1}{\sqrt{p}}\bfH_1\bbeta\right)
\\&-\frac{1}{n\sqrt{p}}(\bfv^\top \bfh \|\bfP^\perp\bbeta\|+\|\bfv\|\bfg^\top\bfP^\perp\bbeta),
\end{split}
\end{equation}
where $\bfh\in\mathbb{R}^n$ and $\bfg\in\mathbb{R}^p$ both have i.i.d. standard normal entries that are independent with the other random quantities.

\subsection*{Analysis of AO}
Next we analyze the AO in (\ref{AO-2}). The goal here is to turn the vector optimization problem into an equivalent form of a scalar optimization problem. We first perform the maximization with respect to the direction of $\bfv$. The terms that are related to $\bfv$ induce the
following maximization with respect to $\bfv$
\begin{equation}\label{AO-v1}
\max_{\bfv\in\mathbb{R}^n}\frac{1}{n}\bfv^\top\left(\bfu-\frac{1}{\sqrt{p}}\bfH_1\bbeta-\frac{1}{\sqrt{p}}\bfh \|\bfP^\perp\bbeta\|\right)-\frac{1}{n\sqrt{p}}\|\bfv\|\bfg^\top\bfP^\perp\bbeta.
\end{equation}
The direction of the optimizer $\bfv^*$ must satisfy that
\begin{align*}
\frac{\bfv^*}{\|\bfv^*\|}=\frac{\bfu-\frac{1}{\sqrt{p}}\bfH_1\bbeta-\frac{1}{\sqrt{p}}\bfh \|\bfP^\perp\bbeta\|}{\left\|\bfu-\frac{1}{\sqrt{p}}\bfH_1\bbeta-\frac{1}{\sqrt{p}}\bfh \|\bfP^\perp\bbeta\|\right\|}.  
\end{align*}
Thus we can write (\ref{AO-v1}) as
\begin{equation}\label{AO-v2}
\max_{r\geq 0}r\left(\frac{1}{n}\left\|\bfu-\frac{1}{\sqrt{p}}\bfH_1\bbeta-\frac{1}{\sqrt{p}}\bfh \|\bfP^\perp\bbeta\|\right\|-\frac{1}{n\sqrt{p}}\bfg^\top\bfP^\perp\bbeta\right),
\end{equation}
where $r=\|\bfv^*\|$. Plugging the above expression into (\ref{AO-2}), we obtain
\begin{equation}\label{AO-3}
\begin{split}
\min_{\bbeta\in\mathbb{R}^p,\bfu\in\mathbb{R}^n}\max_{r\geq 0}&-\frac{1}{n} \bDelta^\top\bfu + \frac{1}{n}\bDelta^\top \log\left(\bfA \exp(\bfu)\right)
\\&+r\left(\frac{1}{n}\left\|\bfu-\frac{1}{\sqrt{p}}\bfH_1\bbeta-\frac{1}{\sqrt{p}}\bfh \|\bfP^\perp\bbeta\|\right\|-\frac{1}{n\sqrt{p}}\bfg^\top\bfP^\perp\bbeta\right).
\end{split}
\end{equation}
As the original optimization problem is convex with respect to $\bbeta$ and $\bfu$ and concave with respect to $\bfv$, in an asymptotic sense, we can flip the maximization with the minimization. We consider the following problem,
\begin{align}
&\min_{\bbeta\in\mathbb{R}^p}\frac{1}{n}\left\|\bfu-\frac{1}{\sqrt{p}}\bfH_1\bbeta-\frac{1}{\sqrt{p}}\bfh \|\bfP^\perp\bbeta\|\right\|-\frac{1}{n\sqrt{p}}\bfg^\top\bfP^\perp\bbeta \nonumber
\\=&\min_{\bbeta\in\mathbb{R}^p}\frac{1}{n}\left\|\bfu-\frac{1}{\sqrt{p}}\bfH\bbeta^*a-\frac{1}{\sqrt{p}}\bfh \|\bfP^\perp\bbeta\|\right\|-\frac{1}{n\sqrt{p}}\frac{\bfg^\top\bfP^\perp\bbeta}{\|\bfP^\perp\bbeta\|}\|\bfP^\perp\bbeta\| \nonumber
\\=&\min_{a\in\mathbb{R},b\geq 0}\frac{1}{n}\left\|\bfu-\frac{1}{\sqrt{p}}\bfH\bbeta^*a-\frac{1}{\sqrt{p}}\bfh b\right\|-\frac{1}{n\sqrt{p}}\|\bfP^\perp \bfg\| b, \label{eq-obj5}
\end{align}
where 
\begin{align*}
a=\frac{\bbeta^\top \bbeta^*}{\|\bbeta^*\|^2}\quad \text{and}\quad b=\|\bfP^\perp\bbeta\|. 
\end{align*}
Plugging (\ref{eq-obj5}) into (\ref{AO-3}) yields that
\begin{equation}\label{AO-4}
\begin{split}
\min_{a\in\mathbb{R},b\geq 0,\bfu\in\mathbb{R}^n}\max_{r\geq 0}&-\frac{1}{n}\bDelta^\top\bfu + \frac{1}{n}\bDelta^\top \log\left(\bfA \exp(\bfu)\right)
\\&+r\left(\frac{1}{n}\left\|\bfu-\frac{1}{\sqrt{p}}\bfH\bbeta^*a-\frac{1}{\sqrt{p}}\bfh b\right\|-\frac{1}{n\sqrt{p}}\|\bfP^{\perp} \bfg\| b\right).
\end{split}
\end{equation}
We notice that for any $s_0>0,$
\begin{align*}
\min_{v\geq 0}\frac{1}{2v}+\frac{v s_0^2}{2}=s_0.    
\end{align*}
Using this fact, we can reformulate (\ref{AO-4})
as 
\begin{equation}\label{AO-5}
\begin{split}
\min_{a\in\mathbb{R},b,v\geq 0,\bfu\in\mathbb{R}^n}\max_{r\geq 0}&-\frac{1}{n}\bDelta^\top\bfu + \frac{1}{n}\bDelta^\top \log\left(\bfA \exp(\bfu)\right)+\frac{r}{2v}
\\&+\frac{rv}{2}\left\|\frac{\bfu}{n}-\frac{1}{n\sqrt{p}}\bfH\bbeta^*a-\frac{1}{n\sqrt{p}}\bfh b\right\|^2-\frac{r}{n\sqrt{p}}\|\bfP^{\perp} \bfg\| b.
\end{split}
\end{equation}
Replacing $v$, $r$ and $b$ by $\sqrt{n}v$, $\sqrt{n}r$ and $\sqrt{p}b$ respectively, we obtain 
\begin{equation}\label{AO-6}
\begin{split}
\min_{a\in\mathbb{R},b,v\geq 0,\bfu\in\mathbb{R}^n}\max_{ r\geq 0}&-\frac{1}{n}\bDelta^\top\bfu + \frac{1}{n}\bDelta^\top \log\left(\bfA \exp(\bfu)\right)+\frac{r}{2v}
\\&+\frac{rv}{2}\left\|\frac{\bfu}{\sqrt{n}}-\frac{1}{\sqrt{n p}}\bfH\bbeta^*a-\frac{1}{\sqrt{n}}\bfh b\right\|^2-\frac{r}{\sqrt{n}}\|\bfP^{\perp} \bfg\| b.
\end{split}
\end{equation}
Some algebra yields that
\begin{align*}
&-\frac{1}{n} \bDelta^\top\bfu+  \frac{rv}{2}\left\|\frac{\bfu}{\sqrt{n}}-\frac{1}{\sqrt{n p}}\bfH\bbeta^*a-\frac{1}{\sqrt{n }}\bfh b\right\|^2
\\=&\frac{rv}{2}\left\|\frac{\bfu}{\sqrt{n}}-\frac{1}{\sqrt{n p}}\bfH\bbeta^*a-\frac{1}{\sqrt{n}}\bfh b-\frac{\bDelta }{rv\sqrt{n}}\right\|^2-\frac{\|\bDelta \|^2}{2nrv}-\frac{ \bDelta^\top \bfH \bbeta^* a}{n\sqrt{p}}-\frac{ \bDelta^\top \bfh b}{n}.
\end{align*}
As $\bfg\in\mathbb{R}^p$ has i.i.d standard normal entries,
we have
\begin{align*}
&\frac{\|\bfP^{\perp} \bfg\|}{\sqrt{p}}\rightarrow^p 1,
\end{align*}
by the law of large numbers. Using Assumption A3, we have
\begin{align*}
&\frac{\bDelta }{n}\rightarrow^p 1-\mathbb{E}[S(C|\kappa Z)], \\  
&\frac{ \bDelta^\top \bfH \bbeta^* }{n\sqrt{p}}=\frac{ \bDelta^\top \bfH \bbeta^* }{n\|\bbeta^*\|}\frac{\|\bbeta^*\|}{\sqrt{p}}\rightarrow^p -\kappa
\mathbb{E}[S(C|\kappa Z)Z],\\
&\frac{\bDelta^\top \bfh}{n}\rightarrow^p 0.
\end{align*}
Combining the above arguments leads to the following optimization problem
\begin{equation}\label{AO-6}
\begin{split}
\min_{a\in\mathbb{R},b,v\geq 0,\bfu\in\mathbb{R}^n}\max_{ r\geq 0}&\frac{1}{n}\bDelta^\top \log\left(\bfA \exp(\bfu)\right)+\frac{r}{2v}
+\frac{rv}{2n}\left\|\bfu-\kappa a\mathbf{q}-b \bfh-\frac{\bDelta }{rv}\right\|^2
\\&-\frac{1-\mathbb{E}[S(C|\kappa Z)]}{2rv}+\kappa a\mathbb{E}[S(C|\kappa Z)Z]-r \sqrt{\delta} b,
\end{split}
\end{equation}
where $\mathbf{q}=(q_1,\dots,q_n)^\top=\bfH \bbeta^*/(\kappa\sqrt{p})$. Let
\begin{align*}
&G_n(\bfu):=\bDelta^\top \log\left(\bfA \exp(\bfu)\right)
=\sum_{i=1}^n\Delta_i\log\left(\frac{1}{n}\sum_{j=1}^n\mathbf{1}\{Y_j\ge Y_i\}\exp(u_j)\right)
\end{align*}
and $\bxi=(\xi_1,\dots,\xi_n)^\top =\kappa a\mathbf{q}+b \bfh+\frac{\bDelta }{rv}.$
Define the Moreau envelope function
\begin{align*}
&\min_{\bfu\in\mathbb{R}^n} G_n(\bfu) +\frac{rv}{2}\left\|\bfu-\bxi\right\|^2=M_{G_n}\left(\bxi;\frac{1}{r v}\right).
\end{align*}
Then (\ref{AO-6}) becomes
\begin{equation}\label{AO-7}
\begin{split}
\min_{a\in\mathbb{R},b,v\geq 0}\max_{r\geq 0}&\frac{1}{n}M_{G_n}\left(\bxi;\frac{1}{r v}\right)
+\frac{r}{2v}
-\frac{1-\mathbb{E}[S(C|\kappa Z)]}{2rv}+\kappa a\mathbb{E}[S(C|\kappa Z)Z]-r \sqrt{\delta}  b.
\end{split}
\end{equation}

\subsection*{Analysis of the Moreau envelope function}\label{sec:M}
Our goal here is to show that as $n\rightarrow \infty,$
\begin{align*}
\frac{1}{n}M_{G_n}\left(\bxi;\frac{1}{r v}\right)   \rightarrow M\left(\kappa a,b,\frac{1}{rv}\right) 
\end{align*}
for some limiting function $M(\cdot,\cdot,\cdot)$. To facilitate the derivations, we introduce some stochastic processes that are useful in the survival analysis \citep{AG82}. Consider an $n$-dimensional counting process $\mathbf{N}^{(n)}(t)=(N_1(t),\dots,N_n(t))$ for $t\geq 0$, where $N_i(t)$ counts the number of observed events for the $i$th individual in the time
interval $[0,1]$. The sample paths of $N_1,\dots,N_n$ are step functions, zero at $t = 0$, with jumps of size $+1$ only. Furthermore, no two components jump at the same time. Let $Y_i(t)\in \{0,1\}$ be a predictable at risk indicator process that can be constructed from data. Note that $N_i(t)$ is a counting process with the intensity process $Y_i(t)\exp(\bfX_i^\top \bbeta^*)\lambda_0(t)$. We can write 
\begin{align*}
\frac{1}{n}G_n(\bfu)=&\frac{1}{n}\sum_{i=1}^n\Delta_i\log\left(\frac{1}{n}\sum_{j=1}^n\mathbf{1}\{Y_j\ge Y_i\}\exp(u_j)\right)
\\=&\int^{1}_{0}\log\left(\frac{1}{n}\sum^{n}_{j=1}Y_j(s)\exp(u_j)\right) d\bar{N}_n(s)
\\ = &\int^{1}_{0}\log\left(\frac{1}{n}\sum^{n}_{j=1}Y_j(s)\exp(u_j)\right)R_n(s;\bbeta^*)\lambda_0(s)ds
\end{align*}
where $\bar{N}_n(t)=n^{-1}\sum^{n}_{i=1}N_i(t)$ and 
$$R_n(s;\bbeta^*)=\frac{1}{n}\sum^{n}_{i=1}Y_i(s)\exp(\bfX_i^\top \bbeta^*)=\frac{1}{n}\sum^{n}_{i=1}Y_i(s)\exp(\kappa q_i).$$
Thus we obtain
\begin{align*}
&\min_{\bfu\in\mathbb{R}^n}\frac{1}{n}G_n(\bfu) +\frac{rv}{2n}\left\|\bfu-\bxi\right\|^2
\\ = &\min_{\bfu\in\mathbb{R}^n} \int^{1}_{0}\log\left(\frac{1}{n}\sum^{n}_{j=1}Y_j(s)\exp(u_j)\right)\frac{1}{n}\sum^{n}_{i=1}Y_i(s)\exp(\kappa q_i)\lambda_0(s)ds
\\&+\frac{rv}{2n}\sum^{n}_{j=1}(u_j-\xi_j)^2.
\end{align*}
The first order condition implies that at the optimal $\bfu^*=(u_1^*,\dots,u_n^*)$
\begin{align*}
rv(u_k^*-\xi_k)=-\int^{1}_{0}\frac{Y_k(s)\exp(u_k^*)}{\frac{1}{n}\sum^{n}_{j=1}Y_j(s)\exp(u_j^*)}R_n(s;\bbeta^*)\lambda_0(s)ds.    
\end{align*}
Squaring both sides, summing over $k$ and scaling both sides by $1/n$, we obtain
\begin{align*}
\frac{r^2v^2}{n}\sum^{n}_{k=1}(u_k^*-\xi_k)^2=\int^{1}_{0}\int^{1}_{0}\frac{\frac{1}{n}\sum^{n}_{k=1}Y_k(s)Y_k(t)\exp(2u_k^*)}{\frac{1}{n^2}\sum^{n}_{i,j=1}Y_i(s)Y_j(t)\exp(u_i^*+u_j^*)}R_n(s;\bbeta^*)R_n(t;\bbeta^*)\lambda_0(s)\lambda_0(t)dsdt.    
\end{align*}
Under Assumption A4, we have
\begin{align*}
\frac{r^2v^2}{n}\sum^{n}_{k=1}(u_k^*-\xi_k)^2\rightarrow^p \int^{1}_{0}\int^{1}_{0}\frac{S(s,t)}{S(s)S(t)}R(s)R(t)ds dt,    
\end{align*}
and 
\begin{align*}
\int^{1}_{0}\log\left(\frac{1}{n}\sum^{n}_{j=1}Y_j(s)\exp(u_j^*)\right)\frac{1}{n}\sum^{n}_{i=1}Y_i(s)\exp(\kappa q_i)\lambda_0(s)ds \rightarrow^p \int^{1}_{0}\log(S(s))R(s)ds.   
\end{align*}
Therefore, we get
\begin{align*}
\frac{1}{n}M_{G_n}\left(\bxi;\frac{1}{r v}\right)\rightarrow^p  
\int^{1}_{0}\log(S(s))R(s)ds+\frac{1}{2rv}\int^{1}_{0}\int^{1}_{0}\frac{S(s,t)}{S(s)S(t)}R(s)R(t)ds dt.
\end{align*}
Next we show how $S(s,t)$ and $S(t)$ depend on $\xi_i$.
Note that
\begin{align*}
\frac{rv(u_k^*-\xi_k)}{\exp(u_k^*)}=-\int^{1}_{0}\frac{Y_k(s)}{\frac{1}{n}\sum^{n}_{j=1}Y_j(s)\exp(u_j^*)}R_n(s;\bbeta^*)\lambda_0(s)ds\rightarrow^p
-\int^{1}_{0}\frac{Y_k(s)}{S(s)}R(s)ds.
\end{align*}
We can solve this nonlinear equation for $u_k^*$ in terms of $rv$, $\xi_k$ and 
$\int^{1}_{0}\frac{Y_k(s)}{S(s)}R(s)ds$.
Asymptotically, $u_k^*$ satisfies the nonlinear equation 
$$\frac{rv(u_k^*-\xi_k)}{\exp(u_k^*)}=
-\int^{1}_{0}\frac{Y_k(u)}{S(u)}R(u)du.$$ We write the solution as
\begin{align*}
u_k^*=\log\left\{K\left(\xi_k,\int^{1}_{0}\frac{Y_k(u)}{S(u)}R(u)du,rv\right)\right\}.   
\end{align*}
Then we have
\begin{align*}
\frac{1}{n}\sum^{n}_{k=1}Y_k(s)\exp(u_k^*)=\frac{1}{n}\sum^{n}_{k=1}Y_k(s) K\left(\xi_k,\int^{1}_{0}\frac{Y_k(u)}{S(u)}R(u)du,rv\right). 
\end{align*}
Letting $n\rightarrow+\infty$, we have $S(\cdot)$ being the solution to the following equation
\begin{align*}
S(s)=E\left[Y(s)K\left(\xi,\int^{1}_{0}\frac{Y(u)}{S(u)}R(u)du,rv\right)\right].    
\end{align*}
Similarly, we have
\begin{align*}
\frac{1}{n}\sum^{n}_{k=1}Y_k(s)Y_k(t)\exp(2u_k^*)=&\frac{1}{n}\sum^{n}_{k=1}Y_k(s)Y_k(t)K^2\left(\xi_k,\int^{1}_{0}\frac{Y(u)}{S(u)}R(u)du,rv\right)
\end{align*}
which implies 
\begin{align*}
S(s,t)=E\left[Y(s)Y(t)K^2\left(\xi,\int^{1}_{0}\frac{Y(u)}{S(u)}R(u)du,rv\right)\right].    
\end{align*}
Combining the above results, we have shown that
\begin{align*}
&M\left(\kappa a,b,\frac{1}{rv}\right) \\  =&\int^{1}_{0}\log(S(s))R(s)ds+\frac{1}{2rv}\int^{1}_{0}\int^{1}_{0}\frac{E\left[Y(s)Y(t)K^2\left(\xi,\int^{1}_{0}\frac{Y(u)}{S(u)}R(u)du,rv\right)\right]}{S(s)S(t)}R(s)R(t)ds dt,  
\end{align*}
where
$S(\cdot)$ is the solution to the equation
\begin{align*}
S(s)=E\left[Y(s)K\left(\xi,\int^{1}_{0}\frac{Y(u)}{S(u)}R(u)du,rv\right)\right]
\end{align*}
with 
$\xi=\kappa a q+b h+\frac{\Delta }{rv}$ and 
$R(s)=\lambda_0(s)E\left[Y(s)\exp(\kappa q)\right].$

\subsection*{Optimality conditions}
Since the objective function is smooth, when the optimal values are all non-zero, they
should satisfy the first order optimality condition. We derive the conditions for $a$, $b$, $v$ and $r$ for the problem 
\begin{equation}\label{AO-8}
\begin{split}
\min_{a\in\mathbb{R},b,v\geq 0}\max_{r\geq 0}&M\left(\kappa a,b,\frac{1}{rv}\right)
+\frac{r}{2v}
-\frac{1-\mathbb{E}[S(C|\kappa Z)]}{2rv}+\kappa a\mathbb{E}[S(C|\kappa Z)Z]-r \sqrt{\delta} b.
\end{split}
\end{equation}
separately below. Let 
\begin{align*}
M_i(a_1,a_2,a_3)=\frac{\partial M(a_1,a_2,a_3)}{\partial a_i},\quad 1\leq i\leq 3.    
\end{align*}
\begin{itemize}
    \item \textbf{Condition for $a$}
    $$M_1\left(\kappa a,b,\frac{1}{rv}\right)
+\mathbb{E}[S(C|\kappa Z)Z]=0.$$
    \item \textbf{Condition for $b$}
    \begin{align*}
    M_2\left(\kappa a,b,\frac{1}{rv}\right) =r \sqrt{\delta} .    
    \end{align*}
    \item \textbf{Condition for $v$}
    \begin{align*}
    -\frac{1}{rv^2}M_3\left(\kappa a,b,\frac{1}{rv}\right)
-\frac{r}{2v^2}
+\frac{1-\mathbb{E}[S(C|\kappa Z)]}{2rv^2}=0.
    \end{align*}
    \item \textbf{Condition for $r$}
    \begin{align*}
    -\frac{1}{r^2v}M_3\left(\kappa a,b,\frac{1}{r v}\right)
+\frac{1}{2v}
+\frac{1-\mathbb{E}[S(C|\kappa Z)]}{2r^2v}-\sqrt{\delta }b=0.   
    \end{align*}
\end{itemize}
The last two equations imply that
\begin{align*}
b=\frac{1}{v\sqrt{\delta}}.    
\end{align*}
Therefore, we obtain the following set of equations
\begin{align*}
&M_1\left(\kappa a,b,\frac{b\sqrt{\delta}}{r}\right)
=-\mathbb{E}[S(C|\kappa Z)Z], \\
&M_2\left(\kappa a,b,\frac{b\sqrt{\delta}}{r}\right) =r\sqrt{\delta},\\
&M_3\left(\kappa a,b,\frac{b\sqrt{\delta}}{r}\right)
=-\frac{r^2}{2}
+\frac{1}{2}\left(1-\mathbb{E}[S(C|\kappa Z)]\right).
\end{align*}

\subsection*{Approximate solution}
Recall that $b\sqrt{\delta}=1/v$. Consider the problem
\begin{equation*}
\min_{a\in\mathbb{R},b\geq 0}\max_{r\geq 0}\frac{1}{n}M_{G_n}\left(\bxi;\frac{b\sqrt{\delta}}{r}\right) -\frac{b\sqrt{\delta}}{2 r}\left(1-\mathbb{E}[S(C|\kappa Z)]\right)+\kappa a\mathbb{E}[S(C|\kappa Z)Z]-\frac{r\sqrt{\delta} b}{2},
\end{equation*}
where
\begin{align*}
&M_{G_n}\left(\bxi;\frac{b\sqrt{\delta}}{r }\right)=\min_{\bfu\in\mathbb{R}^n} \sum_{i=1}^n\Delta_i\log\left(\frac{1}{n}\sum_{j=1}^n\mathbf{1}\{Y_j\ge Y_i\}\exp(u_j)\right) +\frac{r}{2b\sqrt{\delta}}\left\|\bfu-\bxi\right\|^2,\end{align*}
with $$\bxi=\kappa a\frac{\bfH \bbeta^*}{\|\bbeta^*\|}+b \bfh+\frac{\sqrt{\delta} b\bDelta }{r}$$
for $\bfH=\sqrt{p}(\bfX_1,\dots,\bfX_n)^\top.$ We solve the above min-max problem numerically to obtain the approximate solution $(a^*,b^*,r^*)$ to the set of nonlinear equations.

\subsection*{Proof of Theorem \ref{thm-error}}
We provide a sketch of the proof. Inspecting the derivations in Section \ref{sec:error-analysis}, we know that the scalar quantity $a$ results from the transformation
$$a=\frac{\bbeta^\top \bbeta^*}{\|\bbeta^*\|^2},$$
and the quantity $b$ is related to $\bbeta$ through
\begin{align*}
b=\frac{\|\bfP^\perp\bbeta\|}{\sqrt{p}}.
\end{align*}
Let $\widehat{\bbeta}_{\text{AO}}$ be the solution to the AO in (\ref{AO-2}). 
As (\ref{AO-2}) and (\ref{AO-8}) are asymptotically equivalent, we have
\begin{align*}
\frac{\widehat{\bbeta}^\top_{\text{AO}} \bbeta^*}{\|\bbeta^*\|^2}\rightarrow a^* \quad \text{and} \quad
\frac{\|\bfP^\perp\widehat{\bbeta}_{\text{AO}}\|}{\sqrt{p}}\rightarrow b^*.    
\end{align*}
Therefore, we have
\begin{align*}
&\frac{\|\widehat{\bbeta}_{\text{AO}}- \bbeta^*\|^2}{\|\bbeta^*\|^2}
=\frac{\|\bfP\widehat{\bbeta}_{\text{AO}}\|^2+\|\bfP^\perp\widehat{\bbeta}_{\text{AO}}\|^2-2\widehat{\bbeta}^\top_{\text{AO}} \bbeta^*+\|\bbeta^*\|^2}{\|\bbeta^*\|^2}\rightarrow^p (a^*-1)^2+\frac{ (b^*)^2}{\kappa^2},\\
&\frac{\|\widehat{\bbeta}_{\text{AO}}-a^*\bbeta^*\|^2}{p}\rightarrow^p (b^*)^2.
\end{align*}
Now consider the event
\begin{align*}
\mathcal{S}=\left\{\bbeta\in\mathbb{R}^p: \left|\frac{\|\bbeta- \bbeta^*\|^2}{\|\bbeta^*\|^2}-(a^*-1)^2-\frac{ (b^*)^2}{\kappa^2}\right|\leq \epsilon \right\}    
\end{align*}
for any $\epsilon>0.$ We have $P(\widehat{\bbeta}_{\text{AO}}\in\mathcal{S})\rightarrow 1.$ Using (iv) of Theorem \ref{thm1}, we have $P(\widehat{\bbeta}\in\mathcal{S})\rightarrow 1.$ The other result can be proved similarly.


\subsection*{Proof of Theorem \ref{thm-test}}
From the analysis of the AO problem, we have 
$\bfP^\perp\widehat{\bbeta}_{\text{AO}}/\|\bfP^\perp\widehat{\bbeta}_{\text{AO}}\|=\bfP^\perp\bfg/\|\bfP^\perp\bfg\|$. For any fixed $\mathbf{d}\in\mathbb{R}^p$ with $\mathbf{d}^\top \bbeta^*=O(1)$ and $\|\mathbf{d}\|^2=O(1)$, we have
\begin{align*}
\mathbf{d}^\top \widehat{\bbeta}_{\text{AO}}
=& \mathbf{d}^\top\bfP\widehat{\bbeta}_{\text{AO}}+\mathbf{d}^\top\bfP^\perp\widehat{\bbeta}_{\text{AO}}
\\ =& \mathbf{d}^\top\bbeta^*\frac{\widehat{\bbeta}^\top_{\text{AO}}\bbeta^*}{\|\bbeta^*\|^2}+ \frac{\mathbf{d}^\top\bfP^\perp\widehat{\bbeta}_{\text{AO}}}{\|\bfP^\perp\widehat{\bbeta}_{\text{AO}}\|}\|\bfP^\perp\widehat{\bbeta}_{\text{AO}}\|
\\ =& \mathbf{d}^\top\bbeta^*\frac{\widehat{\bbeta}^\top_{\text{AO}}\bbeta^*}{\|\bbeta^*\|^2}+ \frac{\mathbf{d}^\top\bfP^\perp\bfg}{\|\bfP^\perp\bfg\|}\|\bfP^\perp\widehat{\bbeta}_{\text{AO}}\| 
\\=&  (a^*+o_p(1))\mathbf{d}^\top\bbeta^*+(b^*+o_p(1))\mathbf{d}^\top\bfP^\perp\bfg.
\end{align*}
where the orders for the $o_p(1)$ terms are uniform over all $\mathbf{d}$ with $\mathbf{d}^\top \bbeta^*=O(1)$ and $\|\mathbf{d}\|^2=O(1)$. 
Choosing $\mathbf{d}$ to be the standard basis vector corresponding to any $j\in\mathcal{S}_0$ gives $\widehat{\beta}_{\text{AO},j}=(b^*+o_p(1))(\bfP^\perp\bfg)_j.$ Thus
\begin{align*}
\frac{1}{|\mathcal{S}_0|}\sum_{j\in\mathcal{S}_0}\widehat{\beta}_{\text{AO},j}^2=(b^*+o_p(1))^2\frac{1}{|\mathcal{S}_0|}\sum_{j\in\mathcal{S}_0}(\bfP^\perp\bfg)_j^2\rightarrow^p (b^*)^2.   
\end{align*}
Consider the event 
$$\mathcal{S}=\left\{\bbeta\in\mathbb{R}^p:\left|\frac{1}{|\mathcal{S}_0|}\sum_{j\in\mathcal{S}_0}\beta^2_j-(b^*)^2\right|\leq\epsilon\right\}$$
for any $\epsilon>0.$ Using (iv) of Theorem \ref{thm1}, we have $P(\widehat{\bbeta}\in\mathcal{S})\rightarrow 1$. Therefore,
\begin{align*}
\frac{1}{|\mathcal{S}_0|}\sum_{j\in\mathcal{S}_0}\widehat{\beta}_{j}^2\rightarrow^p (b^*)^2.       
\end{align*}
Let $\widehat{\bbeta}_{\mathcal{S}_0}=(\widehat{\beta}_{j})_{j\in\mathcal{S}_0}$. Following similar arguments as in the proof of Theorem 3 in \cite{SurModern2019}, we know that for
$\widehat{\bbeta}_{\mathcal{S}}/\|\widehat{\bbeta}_{\mathcal{S}_0}\|$ has the same distribution as that of $\mathbf{Z}_{\mathcal{S}}/\|\mathbf{Z}\|$, where $\mathbf{Z}=(Z_j)_{j\in\mathcal{S}_0}\in\mathbb{R}^{|\mathcal{S}_0|}$ has i.i.d $N(0,1)$ entries and $\mathbf{Z}_{\mathcal{S}}=(Z_j)_{j\in\mathcal{S}}$. As $\|\widehat{\bbeta}_{\mathcal{S}_0}\|/\|\mathbf{Z}\|\rightarrow b^*$, we obtain $
\widehat{\bbeta}_{\mathcal{S}}/b^*\rightarrow^d N(0,\mathbf{I}_l).$



\end{document}